%% file: BCR_Arxiv.tex
\documentclass{article}
\input{setup_extended}

\begin{document}

\title{On the determination of ischemic regions in the monodomain model of cardiac electrophysiology from boundary measurements}
\author{
  Elena Beretta\footnotemark[1],\
	Cecilia Cavaterra\footnotemark[2],\
  Luca Ratti\footnotemark[3]
}

\footnotetext[1]{NYU Abu Dhabi (UAE),  Politecnico di Milano (Italy) \texttt{eb147@nyu.edu, elena.beretta@polimi.it}}
\footnotetext[2]{Universit\`{a} degli Studi di Milano  (Italy), IMATI-CNR (Italy) \texttt{cecilia.cavaterra@unimi.it}}
\footnotetext[3]{University of Helsinki (Finland) \texttt{luca.ratti@helsinki.fi}}

\maketitle

\begin{abstract}
In this paper we consider the monodomain model of cardiac electrophysiology. After an analysis of the well-posedness of the model we determine an asymptotic expansion of the perturbed potential due to the presence of small conductivity inhomogeneities (modelling small ischemic regions in the cardiac tissue) and use it to detect the anomalies from partial boundary measurements. This is done by determining the topological gradient of a suitable boundary misfit functional. The robustness of the algorithm is confirmed by several numerical experiments.
\end{abstract}

\section{Introduction}
\label{intro}
Cardiac ischemia consists in a restriction of blood supply to the heart tissue usually caused by atherosclerosis or coronary syndrome. The shortage of oxygen may lead to dysfunction of the cell metabolism and eventually to their death. The possible outcomes range from ventricular arrhythmia, fibrillation and ultimately to myocardial infarction. 

The ischemic heart syndrome is the most common cardiovascular disease, and the most common cause of death. 
Hence, the detection of ischemic regions at early stage of their development is of primary importance. This is usually performed by imaging techniques such as echocardiography, gamma ray scintigraphy or magnetic resonance imaging. Nevertheless, the most common test for patients not exhibiting evident symptoms is the electrocardiogram (ECG), which consists in recording electrical impulses across the thorax by means of a set of electrodes. Physicians are often able to identify myocardial ischemia by analysing the evolution of the voltage recorded in the ECG leads although with several technical difficulties (see the mathematical studies in \cite{art:ECG} and \cite{art:gerbeau}).

The approach we propose in this paper is to obtain information regarding the electrical functioning of the tissue (and ultimately the presence of ischemic regions) from the knowledge of the electrical potential on the surface of the heart. 
When employing the potential on the  epicardium, i.e., the external boundary of the heart, this can be considered as a natural extension of the well-known mathematical problem often referred to as \textit{the inverse problem of electrocardiography}, which consists in using the ECG recordings to compute the potential distribution on the epicardium. 
Alternatively, we remark that it is also possible to obtain electrical measurements on the endocardium, namely the internal boundaries of the heart cavities. Although much more invasive than the ECG, intracardiac ECG (iECG) has become a standard of care in patients with symptoms of heart failure, and allows to get a map of the endocardiac potential by means of non-contact electrodes carried by a catheter inside a heart cavity.
We therefore assume that the distribution of the electrical potential is available on some portion of the heart boundary (on the epicardium, in the case of ECG data, or on the endocardium, in the case of iECG measurements). After a reliable model is introduced for the description of the evolution of the electrical potential within the heart, the problem of detecting ischemic regions is hence formulated as an inverse boundary value problem of identifying parameters in a nonlinear reaction diffusion system from boundary data.

In this paper we focus on the determination of small ischemias from boundary measurements of the electrical potential, generalizing the results obtained in previous papers (see \cite{art:BCMP},\cite{art:BMR},\cite{art:bccmr}) to the monodomain time-dependent model of cardiac electrophysiology.

In general, this model is described by a system consisting in a semilinear parabolic equation coupled with a system of nonlinear ordinary differential equations, where the state variables are the transmembrane potential across the cell membrane, the concentrations of ionic species and the gating variables describing the activity of ionic channels in the membrane.
It is well known (see \cite{book:pavarino} and \cite{art:boulakia}) that the monodomain model can be derived from the more complex bidomain one, introduced for the first time in \cite{phd:tung}, for instance by assuming proportionality between intracellular and extracellular conductivity tensors. In particular, in this case, the corresponding conductivity tensor becomes the harmonic mean of the extra and intracellular conductivities. On the other hand, this choice of the conductivity tensor turns out to be the best one to approximate the electrical propagation described by the bidomain model (see, e.g., \cite{monovsbi1},\cite{monovsbi2}).    

As a first attempt, we limit ourselves to analyze a coupled system of two equations in the potential $u$ and the recovery
variable $w$. This corresponds to a large class of phenomenological models, which are characterized by the choice of the nonlinear terms appearing in the partial differential equation and in the ordinary differential equation. Among the most important two-equation systems (see \cite{book:pavarino} for a general overview)
we mention FitzHugh-Nagumo, Rogers-McCulloch and Aliev-Panfilov models. Throughout this paper, we focus on the commonly used version of the Aliev-Panfilov model (originally introduced in \cite{art:alpanf}), even though the analysis could be extended also to the other two.
\par
In order to detect ischemic regions, we extend the approach of \cite{art:bccmr}, determining a rigorous asymptotic expansion for the perturbed boundary potential due to small conductivity anomalies. To accomplish this task, we need an accurate analysis of the well-posedness of the direct problem for the coupled
system in the case of discontinuous anisotropic coefficients and suitable regularity estimates for solutions. In particular, we establish a comparison principle for this class of systems, which to our knowledge was not present in the literature.
Here, we consider the case of an insulated heart and we assume to have measurements of the potential on a portion of the boundary.
\par
The theory of detection of small conductivity inhomogeneities from boundary measurements via asymptotic techniques has been developed in the last three decades in the framework of Electrical Impedance Tomography (see, e.g., \cite{book:ammari-kang},\cite{art:cfmv},\cite{art:ammari2012}). A similar approach has also been used in Thermal Imaging (see, e.g., \cite{art:aikk}).
Here, we are able to extend in a non trivial way the results obtained previously in \cite{art:BMR} and \cite{art:bccmr} for simplified versions of the monodomain model making use of fine regularity for the solutions to nonlinear reaction diffusion systems. In particular, we establish a rigorous expansion for the perturbed transmembrane potential and use this to implement an effective reconstruction algorithm.
\par
The paper is organized as follows. In Section \ref{direct} we analyze the well-posedness of the forward problem in the unperturbed and perturbed case.
Section \ref{inverse} is devoted to obtaining energy estimates on the difference between the perturbed and unperturbed electrical potential and an asymptotic expansion of suitable integral terms involving such difference on the boundary. In Section \ref{algorithm} we describe our topological-based optimization algorithm and derive rigorously the topological gradient of a suitable mismatch functional. This is obtained by using the results of Section \ref{inverse} and some interior regularity results for the solution of parabolic systems. Finally, in Section \ref{results} we outline the numerical implementation of the proposed algorithm, relying on the Finite Element Method for the discrete formulation of the problem. A significant set of numerical experiments is provided in order to assess the effectiveness of the reconstruction, even in presence of data noise.

\section{Analysis of the direct problem}
 \label{direct}
The well-posedness analysis of the monodomain system has been object of several studies.  We refer to \cite[Chapter 3]{book:pavarino} for a general overview. In \cite{art:BCP} a result of existence and uniqueness of weak solution is proved for
the FitzHugh-Nagumo, the Aliev-Panfilov and the Rogers-MacCulloch models by means of a Faedo-Galerkin procedure. A result of existence of strong solutions, local in time, is also derived. In \cite{art:veneroni}, instead, results of well-posedness are obtained for a wider range of models, on the base of a fixed point argument.
\par
Regarding the regularity of the solutions of the monodomain system, we report a result in \cite{pierrecoudiere} for FitzHugh-Nagumo, Aliev-Panfilov and Rogers-MacCulloch models: if no discontinuities are present in the coefficients of the system, existence and uniqueness of strong solutions is guaranteed, locally in time (see for instance \cite{book:smoller} and \cite{book:henry}). A comparison principle is also provided, by means of the tool of invariant sets, allowing to prove existence of global solutions. We also report a result of local existence of classical solutions for the bidomain model, recently obtained in \cite{giga}.
\par
In this section we focus on the monodomain system in the case of smooth diffusion coefficient and reaction term, corresponding to the case of the healthy tissue
(unperturbed case) and in the case of discontinuities in the diffusion coefficient and in the reaction term, corresponding to the presence of an ischemia in the heart
tissue (perturbed case). We state an existence, uniqueness and comparison result for classical solutions of the unperturbed case (a proof of
which, alternative to the approach of \cite{pierrecoudiere}, is proposed in \cite[Chapter 6]{phd:Luca}) and prove a result regarding existence, uniqueness and regularity of weak solutions in the perturbed case.
\par
In particular, the initial and boundary value problem associated to the unperturbed monodomain system is the following one
\begin{equation}
\left\{
\begin{aligned}
\partial_t u - div(K_0\nabla u) + f(u,w) &= 0 \qquad &\text{in } Q_T, \\
K_0 \nabla u \cdot \normal &= 0 \qquad &\text{on }  \Gamma_T, \\
\partial_t w + g(u,w) &= 0 \qquad &\text{in }  Q_T, \\
u(\cdot,0) = u_0 \qquad w(\cdot,0) &= w_0 \qquad &\text{in }  \Omega,
\end{aligned}
\right.
\label{eq:formaforte}
\end{equation}
where $\Omega$ is the region occupied by the hearth tissue, $\nu$ is the outward unit normal vector to the boundary $\partial\Omega$, $Q_T := \Omega \times (0,T)$ and $\Gamma_T: \partial\Omega \times (0,T)$. A slightly different formulation of the monodomain model (see, e.g., \cite{book:pavarino}) involves the presence of a source term in the right-hand side of the first equation in \eqref{eq:formaforte}, representing an applied current, during a limited time window (the initial activation of the tissue). We replace the effect of such a current with the presence of a non-null initial value $u_0$. Since the modeling differences between the two formulations are negligible after the first instants, we choose the one proposed in \eqref{eq:formaforte} more suited for the mathematical analysis of the problem.

In presence of an ischemia $\omega\subset\Omega$, the perturbed case is described by the model
\begin{equation}
\left\{
\begin{aligned}
\partial_t \uo - div(K_\omega \nabla \uo) + (1-\chi_\omega)f(\uo,\wo) &= 0 \qquad &\text{in } Q_T, \\
K_\omega \nabla \uo \cdot \normal &= 0 \qquad &\text{on }  \Gamma_T, \\
\partial_t \wo + g(\uo,\wo) &= 0 \qquad &\text{in }  Q_T, \\
\uo(\cdot,0) = u_0  \qquad \wo(\cdot,0) &= w_0 \qquad &\text{in }  \Omega,
\end{aligned}
\right.
\label{eq:perturbed}
\end{equation}
where  $\chi_\omega$ is the indicator function of $\omega$ and $K_\omega = K_0 - (K_0-K_1)\chi_\omega$. We now specify the requirements on the domain, the coefficients and the source terms.
\begin{assumption}

\begin{enumerate}
	\item $\Omega \subset \R^d$ bounded domain, $d=2,3$, and $\partial \Omega \in C^{2+\alpha}$;
	
	\item the inclusion $\omega\subset\Omega$ is well separated from the boundary, i.e., 
	\begin{equation}
	\vspace{-0.25cm}
	\exists \ \mathcal{C}_0 \text{ compact subset of } \Omega \text{ s.t. } \omega \subset \mathcal{C}_0 \text{ and } dist(\mathcal{C}_0,\partial\Omega) \geq d_0 >0.
	\vspace{-0.25cm}
	\label{eq:separ}
\end{equation}
	\vspace{-0.25cm}
	\item $K_0, K_1 \in C^{2}(\overline{\Omega};\R^{d \times d})$ are symmetric matrix-valued functions in $\Omega$;
	$\forall x \in \overline{\Omega}$, the matrices $K_0(x)$ and $K_1(x)$ admit $d$ positive eigenvalues $k_{0,1} \leq \ldots \leq k_{0,d}$ and $k_{1,1} \leq \ldots \leq k_{1,d}$ respectively, associated to the same eigenvectors $\vec{e}_1(x),\ldots \vec{e}_d(x)$ such that $k_{1,i} \leq k_{0,i}$ $\forall i = 1,\ldots,d$;
	\item $u_0 \in C^{2+\alpha}(\overline{\Omega})$, $w_0 \in C^2(\overline{\Omega})$, and $K_0 \nabla u_0 \cdot \normal = 0$
on $\partial\Omega$;
	\item we assume $f,g$ as in the Aliev-Panfilov model, namely:
	\begin{equation}
	\vspace{-0.25cm}
f(u,w) = Au(u-a)(u-1) + uw \qquad g(u,w) = \epsilon (Au(u-1-a) +  w),
\label{eq:AP}
\end{equation}
being $A, \epsilon > 0$ and $0<a<1$.
\end{enumerate}
\label{ass:1}
\end{assumption}
\begin{remark}
The requirements on the matrix-valued functions $K_0$ and $K_1$ are satisfied by the conductivity tensors prescribed 
in the model under consideration (see \cite{book:pavarino,book:sundes-lines}). According to experimental evidence,
the cardiac tissue can be modeled as an orthotropic material, characterized by the presence of fibers and sheets, which define the conductivity eigenvectors. Moreover,
the presence of an ischemia does not affect the direction of the fibers, but reduces the value of all the associated eigenvalues. 
From now on, we will indicate by $k_{min}$ the minimum eigenvalue of $K_1$ and by $k_{max}$ the maximum eigenvalue of $K_0$.
\end{remark}
\begin{remark}
As an immediate consequence of \eqref{eq:AP}, $f$ and $g$ satisfy
the \textit{Tangency condition} on the rectangle $S \defeq [0,1]\times [0,\frac{A(a+1)^2}{4}]$, (see \cite{amann}), i.e., indicating by $\vec{p}$ a generalized outward normal on $\partial S$ ( for all $(\xi_1,\xi_2) \in \partial S$,
$\vec{p}(\xi_1,\xi_2)$ is such that $\vec{p}(\xi_1,\xi_2)\cdot (\xi_1,\xi_2) \geq \vec{p}(\xi_1,\xi_2)\cdot (\eta_1,\eta_2)$ $\forall (\eta_1,\eta_2) \in S$)
then,
\begin{equation}
		\vec{p}(\xi_1,\xi_2) \cdot \left( \begin{aligned} -f(\xi_1,\xi_2) \\ -g(\xi_1,\xi_2)\end{aligned}\ \right) \leq 0 \qquad \forall (\xi_1,\xi_2) \in \partial S.
	\label{eq:nagumo}
	\end{equation}
Moreover, the functions $f,g$ are Lipschitz continuous on $S$ with constants $L_f, L_g \leq L$.
\end{remark}
\begin{remark}
For the sake of brevity, in all the formulas we avoid to indicate time and space integration variables with respect to the classical Lebesgue measure,
unless it is necessary.
\end{remark}
We now outline the main results regarding the well-posedness of the problems \eqref{eq:formaforte} and \eqref{eq:perturbed}.
\begin{theorem}[Unperturbed problem]
Let Assumption \ref{ass:1} holds, and suppose that, $\forall x \in \overline{\Omega},$ $(u_0(x),w_0(x)) \in S$.
Then, problem \eqref{eq:formaforte} admits a unique classical solution $(u,w)$, namely $u \in C^{2+\alpha,1+\alpha/2}(\overline{Q_T})$,
$w \in C^{\alpha,1+\alpha/2}(\overline{Q_T})$. Moreover, $(u(x,t),w(x,t)) \in S$, for each $(x,t) \in \overline{Q_T}$.
\label{th:1}
\end{theorem}
The proof of Theorem \ref{th:1} is derived by means of classical fixed point argument (see \cite[Chapter 8, Sections 9 and 11]{book:pao}).
A detailed proof is reported in \cite[Theorem 6.1]{phd:Luca}.

Regarding the perturbed problem \eqref{eq:perturbed}, we note that, although the conductivity tensor and the nonlinear term are discontinuous, we can extend the results obtained in \cite{art:BCP} thanks to the uniform ellipticity to the boundedness of the conductivity tensors and to the form of the reaction term, deriving the following existence and uniqueness result.
\begin{theorem}[Perturbed problem]
Let the Assumption \ref{ass:1} holds, and suppose that, $\forall x \in \overline{\Omega},$ $(u_0(x),w_0(x)) \in S$. Then, problem \eqref{eq:perturbed} admits
a unique weak solution, 
i.e., a couple $(\uo,\wo)$ such that $\uo \in L^2(0,T;H^1(\Omega))\cap L^\infty(0,T;L^2(\Omega))$, $\partial_t \uo \in L^2(0,T;H^*)$,
$\wo \in L^\infty(0,T;L^2(\Omega))$, $\partial_t \wo \in L^2(Q_T)$ and satisfying, for a.e. $t \in (0,T)$,
\begin{equation}
\begin{aligned}
\langle \partial_t \uo, \varphi \rangle_* + \int_\Omega K_\omega \nabla \uo \cdot \nabla \varphi + \int_\Omega (1-\chi_\omega)f(\uo,\wo) \varphi &= 0
\qquad &\forall \varphi \in H^1(\Omega),\\
\int_\Omega \partial_t \wo \psi + \int_\Omega g(\uo,\wo)\psi &= 0 \qquad &\forall \psi \in L^2(\Omega).
\end{aligned}
\label{eq:weak}
\end{equation}
Moreover, $\uo \in C^{\alpha,\alpha/2}(\overline{Q_T})$, $\wo \in C^{\alpha,\alpha/2}(\overline{Q_T})$ and $(\uo(x,t),\wo(x,t)) \in S$,  $\forall \, (x,t) \in \overline{Q_T}$.

\label{th:2}
\end{theorem}

\begin{proof}
 To start with, note that  uniqueness of the weak solution of \eqref{eq:perturbed} has been shown in the case of Aliev-Panfilov model in \cite[Theorem 1.3]{art:kunisch-wagner} as a
byproduct of a stability result obtained by exploiting the specific nonlinear expression of $f$ and $g$.

We proceed now introducing a sequence of regularized problems $(P_k), \, k\in\N,$ of \eqref{eq:perturbed} and showing that the sequence
of their solutions converges to a weak solution of \eqref{eq:perturbed}. We then exploit additional properties inherited by the approximation process to conclude
the stated regularity results. The uniqueness argument is briefly sketched.
\par
Since $\chi_\omega$ is an indicator function, surely $\chi_\omega \in L^2(\Omega)$; by density arguments and according to \eqref{eq:separ},
\begin{equation}
\exists \{\phi_k\}\subset C^2_c(\Omega): \quad 0 \leq \phi_k(x) \leq 1 \quad \forall x \in \Omega, \quad \phi_k \rightarrow \chi_\omega \text{ in $L^2(\Omega)$ and a.e.}, \label{eq:ass2}
\end{equation}
being $C^2_c(\Omega)$ the space of $C^2$ functions with compact support in $\Omega$.
Define $K_k = K_0 + (K_1-K_0)\phi_k$ and let $(u_k,w_k)$ be the solution of the following problem
\begin{equation}
\left\{
\begin{aligned}
\partial_t u_k - div(K_k \nabla u_k) + (1-\phi_k)f(u_k,w_k) &= 0 \qquad &\text{in } Q_T, \\
K_0 \nabla u_k \cdot \normal &= 0 \qquad &\text{on }  \Gamma_T, \\
\partial_t w_k + g(u_k,w_k) &= 0 \qquad &\text{in }  Q_T, \\
u_k(\cdot,0) = u_0  \qquad w_k(\cdot,0) &= w_0 \qquad &\text{in }  \Omega,
\end{aligned}
\right.
\label{eq:regularized}
\end{equation}
where we have used the fact that $K_k = K_0$ on $\partial \Omega$.
We observe that, for any fixed $k$, an application of Theorem \ref{th:1} ensures the existence and uniqueness of a classical solution of Problem \eqref{eq:regularized}.
Moreover by \eqref{eq:ass2} we have that $(1-\phi_k)f$ and $g$ satisfy the Tangency condition on $S$.
Hence, from Theorem \ref{th:1} we deduce  $(u_k,w_k)\in S$, for all $k$.
\par
We now prove that from $\phi_k \xrightarrow{L^2} \chi_\omega$ the convergence of $(u_k,w_k)$ to a weak solution $(u,w)$ of \eqref{eq:perturbed} holds.
We start by proving some \textit{a priori} estimates. Consider the weak form of the problem solved by $(u_k,w_k)$ and take the classical solutions $u_k$, $w_k$
as test functions. Then, we get
	\[
	\begin{aligned}
	&\half \frac{d}{dt}\left(\norm{L^2(\Omega)}{u_k(\cdot,t)}^2 + \norm{L^2(\Omega)}{w_k(\cdot,t)}^2 \right)
      + \int_\Omega K_k\nabla u_k(\cdot,t)\cdot \nabla u_k(\cdot,t) \\
	& \quad = -\int_\Omega (1-\phi_k)f(u_k(\cdot,t),w_k(\cdot,t))u_k(\cdot,t)- \int_\Omega  g(u_k(\cdot,t),w_k(\cdot,t))w_k(\cdot,t).
	\end{aligned}
	\]
Recall now that $k_{min}$ is the minimum eigenvalue of $K_1$, whereas $k_{max}$ is the maximum eigenvalue of $K_0$.
Moreover, since $\phi_k$, $u_k$, $w_k$ are uniformly bounded independently of $k$ (indeed, $\phi_k(x)\in [0,1]$ and $(u_k,w_k)\in S$) and $f,g$ are continuous,
we can introduce
$M_f := max_{(x,t)\in Q_T} \left|(1-\phi_k)f(u_k,w_k)\right|$ and $M_g:= max_{(x,t)\in Q_T} \left|g(u_k,w_k)\right|$, which are independent of $k$.
Hence, by Young inequality,
	\begin{equation}\label{Young}
	\begin{aligned}
		&\half \frac{d}{dt}\left(\norm{L^2(\Omega)}{u_k(\cdot,t)}^2 + \norm{L^2(\Omega)}{w_k(\cdot,t)}^2 \right)
+ k_{min} \norm{L^2(\Omega)}{\nabla u_k(\cdot,t)}^2 \\
		&\leq \half \left(\norm{L^2(\Omega)}{u_k(\cdot,t)}^2 + \norm{L^2(\Omega)}{w_k(\cdot,t)}^2 \right)
+ \half |\Omega| (M^2_f + M^2_g).
	\end{aligned}
\end{equation}
 Using Gronwall's inequality, we get
\[
\begin{aligned}
	&\left(\norm{L^2(\Omega)}{u_k(\cdot,t)}^2 + \norm{L^2(\Omega)}{w_k(\cdot,t)}^2 \right) \\
	&\quad \leq \left( \norm{L^2(\Omega)}{u_0}^2 + \norm{L^2(\Omega)}{w_0}^2 +  |\Omega|(M^2_f+M^2_g)t \right)e^{t},
\end{aligned}
\]
which implies  
\[
\begin{aligned}
& \norm{L^\infty(0,T;L^2(\Omega))}{u_k}^2+ \norm{L^\infty(0,T;L^2(\Omega))}{w_k}^2 \\
&\leq \left( \norm{L^2(\Omega)}{u_0}^2 + \norm{L^2(\Omega)}{w_0}^2 +  |\Omega|(M^2_f+M^2_g)T \right)e^{T} \defeq c_1^2.
\end{aligned}
\]
Integrating from $0$ to $t$ (\ref{Young}) and using last estimate it also follows that
\[
\norm{L^2(0,T,H^1(\Omega))}{u_k}^2 \leq c_1^2T + \frac{1}{2 k_{min}}\left((|\Omega|(M^2_f+M^2_g)+c_1^2)T+\norm{L^2(\Omega)}{u_0}^2 + \norm{L^2(\Omega)}{w_0}^2\right) \eqdef c_2^2.
\]
A bound for the $H^*$ norm of $\partial_t u$ can be found by considering that, for each $\varphi \in H^1(\Omega)$,
\[
\begin{aligned}
\left|\langle \partial_t u_k(\cdot,t), \varphi \rangle_*\right| &\leq k_{max}\norm{L^2(\Omega)}{\nabla u_k(\cdot,t)}\norm{L^2(\Omega)}{\nabla \varphi}
+ M_f |\Omega|^\half \norm{L^2(\Omega)}{\varphi} \\
\end{aligned}
\]
and computing the $L^2$ norm in time
\[
\norm{L^2(0,T,H^*)}{\partial_t u_k} \leq (k_{max}c_2+M_f  |\Omega|^\half) \eqdef c_3.
\]
Analogously, one proves that $\norm{L^2(Q_T)}{\partial_t w_k}\leq c_4$, with $c_4$ independent of $k$.
\par
As a consequence of the uniform bounds we can ensure that (up to a subsequence)
$\exists u \in L^2(0,T;H^1(\Omega)) \cap L^\infty(0,T;L^2(\Omega))$, $\exists w \in L^\infty(0,T;L^2(\Omega))$, $\exists u^* \in L^2(0,T;H^*)$,
$\exists w^* \in L^2(Q_T)$ such that
\[
	u_k \xrightharpoonup{L^2(0,T,H^1)} u, \quad \partial_t u_k \xrightharpoonup{L^2(0,T,H^*)} u^*, \quad w_k \xrightharpoonup{L^2(Q_T)} w,
\quad \partial_t w_k \xrightharpoonup{L^2(Q_T)} w^*.
\]
Using the result contained in \cite[Theorem 8.1]{book:robinson} and the uniqueness of the weak solution of the perturbed problem we have that $u_k \xrightarrow{L^2(Q_T)} \uo$ (see \cite[Theorem 8.1]{book:robinson}),
hence $u_k \rightarrow \uo$ a.e. in $Q_T$. Moreover, also $w_k \rightarrow \wo$ a.e. in $Q_T$ as it can be straightforwardly obtained by the expression
\begin{equation}
w_k(x,t) = e^{-\epsilon t}w_0(x) + \epsilon A e^{-\epsilon t}\int_0^t ((1+a)u_k - u_k^2)e^{\epsilon s} ds.
\label{eq:solwk}
\end{equation}
Consider now the weak form of the problem solved by $(u_k,w_k)$: $\forall \varphi \in H^1(\Omega)$, $\psi \in L^2(\Omega)$.
\begin{equation}
\begin{aligned}
& \langle \partial_t u_k, \varphi \rangle_* + \int_\Omega K_k \nabla u_k \cdot \nabla \varphi + \int_\Omega(1-\phi_k) Au_k(u_k-a)(u_k-1) \varphi \\
& \quad + \int_\Omega (1-\phi_k)u_k w_k \varphi + \int_\Omega \partial_t w_k \psi
+ \int_\Omega A\epsilon(u_k^2 - (1+a)u_k) \psi + \int_\Omega \epsilon w_k \psi = 0.
\label{eq:regularizedweak}
\end{aligned}
\end{equation}
Taking the limit in $k$, exploiting the weak convergence of $u_k$, $w_k$, $\partial_t u_k$, $\partial_t w_k$, we obtain that
$u^* = \partial_t \uo$, $w^* = \partial_t \wo$ in $\mathcal{D}'(0,T)$ and that the limit $(\uo,\wo)$ satisfies, $\forall \varphi \in H^1(\Omega)$, $\psi \in L^2(\Omega)$,
\begin{equation}
\begin{aligned}
& \langle \partial_t \uo, \varphi \rangle_* + \int_\Omega K_\omega \nabla \uo \cdot \nabla \varphi + \int_\Omega(1-\chi_\omega) f(\uo,\wo) \varphi \\
& \quad + \int_\Omega \partial_t \wo \psi + \int_\Omega g(\uo,\wo) \psi = 0.
\label{eq:weakuw}
\end{aligned}
\end{equation}
Indeed, the convergence of the terms involving the time derivatives is a direct consequence of the weak convergence of $\partial_t u_k$, $\partial_t w_k$ and
of the definition of distributional derivative. The limit of the nonlinear reaction terms can be proved by taking advantage of the dominated convergence theorem and of
the pointwise (a.e.) convergence of $u_k$ and $\phi_k$. The convergence of the diffusion term is obtained by combining the weak convergence of
$u_k$ in $H^1(\Omega)$, the pointwise (a.e.) convergence of $\phi_k$ and the uniform $L^\infty(\Omega)$ bound on $\phi_k$. According to \eqref{eq:weakuw}, we can ensure that the limit $(\uo,\wo)$ is the weak solution of \eqref{eq:perturbed}.
\par
The weak solution $(\uo,\wo)$ is moreover a pointwise (a.e.) limit of the regularized solutions $(u_k,w_k)$. As a consequence, the uniform bound on $(u_k,w_k)$ is
valid also for the limit: $(\uo(x,t),\wo(x,t)) \in S$ a.e. in $Q_T$. This allows to prove the additional H\"older regularity of $\uo$. Indeed, since $K_\omega \in L^{\infty}(\Omega)$, $f(\uo,\wo) \in L^{\infty}(Q_T)$, we can apply
Theorem 10.1 of \cite[Chapter 3]{book:lady} on the first equation in \eqref{eq:perturbed}
$$	\partial_t \uo - div(K_\omega \nabla \uo) = -(1-\chi_\omega) f(\uo,\wo)
$$
to get $\uo \in L^{\infty}(Q_T)$. Now, we can extend the regularity result up to the boundary
due to the hypothesis on $\partial \Omega$ and on $u_0$ contained in Assumption \ref{ass:1}, and conclude $\uo \in C^{\alpha,\alpha/2}(\overline{Q_T})$.
Using  the analytic expression of $\wo$ that can be obtained by \eqref{eq:solwk} and the regularity of $\uo$, we can also deduce that $\wo \in C^{\alpha,1+\alpha/2}(\overline{Q_T})$.
\end{proof}

\par
\section{Analysis of the inverse problem}
\label{inverse}
We now tackle the inverse problem of identifying a perturbation $\omega$ from boundary measurements. Suppose to know
$u_{meas}$, the trace of the solution of \eqref{eq:perturbed} in presence of an unknown inclusion.
The inverse problem reads
\begin{equation}
\textit{find $\omega \subset \Omega$ s.t. } u_\omega|_{\partial \Omega} = u_{meas}.
\label{eq:inv}
\end{equation}
Although the analysis of the direct problem has been performed in presence of an arbitrary inclusion $\omega \subset \Omega$, for the purpose of solving
the inverse problem and  derive a reconstruction algorithm, we limit ourselves at considering the case of inclusions of small size, in analogy to what
done in \cite{art:BCMP}, \cite{art:bccmr}. In particular, we consider a family of inclusions $\omega_\varepsilon$ satisfying \eqref{eq:separ} for each $\varepsilon$ and such that 
\begin{equation}
|\omega_\varepsilon| \rightarrow 0 \text{ as } \varepsilon \rightarrow 0.
\label{eq:small}
\end{equation}

We define $\chi_\varepsilon$ the indicator function of $\omegae$, $K_\varepsilon :=K_0 - (K_0-K_1)\chie$ and $(\ue,\we)$ the solution of
problem \eqref{eq:perturbed} with $\omega = \omegae$, in the sense of Theorem \ref{th:2}.
\subsection{Energy estimates}
We now derive some \textit{energy estimates}  for the difference between $(\ue,\we)$ and $(u,w)$, the solution of the unperturbed
problem \eqref{eq:formaforte} in terms of $|\omega_\varepsilon|$ when $\varepsilon \rightarrow 0$.

From now on we will indicate by $C$ a positive constant depending on the data, independent of $\varepsilon$ and that may vary also in the same line. 
\begin{proposition}
Under Assumption 2.1 the following inequalities hold

$\displaystyle \norm{L^\infty(0,T,L^2(\Omega))}{\ue- u} + \norm{L^\infty(0,T,L^2(\Omega))}{\we- w} \leq C |\omegae|^{\half}$,
\vskip 0.2truecm
$\displaystyle  \norm{L^2(0,T,H^1(\Omega))}{\ue- u} \leq C |\omegae|^{\half}$,
\vskip 0.2truecm
$\displaystyle \norm{L^2(Q_T)}{\ue - u} \leq C |\omegae|^{\half + \eta}, \quad \text{ for some }  \eta > 0$.
\label{prop:energy}
\end{proposition}
\begin{proof}
We consider \eqref{eq:weak} with $\omega=\omegae$ and the weak formulation of the unperturbed problem \eqref{eq:formaforte}.  Subtracting term by term and defining
$\Ue = \ue - u$ and $\We = \we - w$, we have that $\Ue(\cdot,0)= 0$, $\We(\cdot,0)=0$ and, for almost every $t\in (0,T)$,
\begin{equation}
	\begin{aligned}
	&\int_\Omega \partial_t \Ue \varphi + \int_\Omega K_\varepsilon \nabla \Ue \cdot \nabla \varphi + \int_\Omega (1-\chi_\varepsilon)(f(\ue,\we)-f(u,w))\varphi
+ \int_\Omega \partial_t \We \psi \\
	&+ \int_\Omega (g(\ue,\we)-g(u,w))\psi = \int_{\omegae} (K_0- K_1) \nabla u\cdot \nabla \varphi + \int_{\omegae} f(u,w) \varphi.
	\end{aligned}
\label{eq:mn}
\end{equation}
Let $\varphi = \Ue$, $\psi = \We$ then, according to Theorem \ref{th:1} and Theorem \ref{th:2}, both $(u,w)$ and $(\ue,\we)$ range within the rectangle $S$, on
which
the functions $f$ and $g$ are Lipschitz continuous with constants less or equal than $L$. Hence
\[
\begin{aligned}
& \half \frac{d}{dt}\left(\norm{L^2(\Omega)}{\Ue}^2 + \norm{L^2(\Omega)}{\We}^2\right)
+ k_{min} \norm{L^2(\Omega)}{\nabla \Ue}^2 \leq 2L\left( \norm{L^2(\Omega)}{\Ue}^2 + \norm{L^2(\Omega)}{\We}^2\right) \\
& \quad + \int_{\omegae} (K_0- K_1) \nabla u\cdot \nabla \Ue + \int_{\omegae} f(u,w) \Ue.
\end{aligned}
\]
Via an application of Schwarz and Young inequalities on the last two terms in the right-hand side of the previous inequality, and from the regularity of the
solution $(u,w)$, we can deduce
\begin{equation}
\begin{aligned}
\half \frac{d}{dt}\left(\norm{L^2(\Omega)}{\Ue}^2 + \norm{L^2(\Omega)}{\We}^2\right) &+\frac{k_{min}}{2} \norm{L^2(\Omega)}{\nabla \Ue}^2\\
&\leq C \left( \norm{L^2(\Omega)}{\Ue}^2 + \norm{L^2(\Omega)}{\We}^2\right) + C |\omegae|.
\label{eq:ener}
\end{aligned}
\end{equation}
An application of Gronwall's lemma to \eqref{eq:ener} entails that
$\norm{L^\infty(0,T;L^2(\Omega))}{\Ue}^2$, $\norm{L^\infty(0,T;L^2(\Omega))}{\We}^2$ and ultimately $\norm{L^2(0,T;H^1(\Omega))}{\Ue}^2$ can be bounded by
$C |\omegae|$. This allows to conclude the first two statements.
\par
Observe now that the pair $(\Ue,\We)$ is also the solution of
\begin{equation}
\left\{
\begin{aligned}
&\partial_t \Ue - div(K_0 \nabla \Ue) + (1-\chi_\varepsilon)(f(\ue,\we)-f(u,w)) =\\
& \qquad -div((K_0-K_1)\chi_\varepsilon \nabla \ue) + \chi_\varepsilon f(u,w) \\
&\partial_t \We + g(\ue,\we) - g(u,w) = 0.
\end{aligned}
\right.
\label{eq:mn2}
\end{equation}
By the mean value theorem, there exist two pairs of functions $(u_{\xi_1}, w_{\xi_1})$, $(u_{\xi_2}, w_{\xi_2})$ s.t.
\footnote{it follows by an application of the Lagrange's mean value theorem on the real-valued function
$h(\tau) = \int_\Omega (1-\chi_\varepsilon)(f(u+\tau(\ue - u), w+\tau(\we - w))-f(u,w))\Ue$ on the interval $[0,1]$; the same holds for $g$.}
\[
\begin{aligned}
\int_\Omega (1-\chi_\varepsilon)(f(\ue,\we)-f(u,w))\Ue &= \int_\Omega (1-\chi_\varepsilon)f_u(u_{\xi_1}, w_{\xi_1})\Ue^2
+ \int_\Omega (1-\chi_\varepsilon)f_w(u_{\xi_1}, w_{\xi_1})\Ue\We \\
\int_\Omega (g(\ue,\we)-g(u,w))\We &= \int_\Omega g_u(u_{\xi_2}, w_{\xi_2})\Ue\We + \int_\Omega g_w(u_{\xi_2}, w_{\xi_2})\We^2.
\end{aligned}
\]
By definition,  $(u_{\xi_1}, w_{\xi_1})$ and $(u_{\xi_2}, w_{\xi_2})$ are convex combinations of $(u,w)$ and $(\ue,\we)$, thus they assume values in
the rectangle $S$. Let now $(\overline{U}_\varepsilon,\overline{W}_\varepsilon)$ be the solution of the adjoint problem
\begin{equation}
\left\{
\begin{aligned}
\partial_t \overline{U}_\varepsilon + div(K_0 \nabla \overline{U}_\varepsilon)
- (1-\chi_\varepsilon)f_u(u_{\xi_1},w_{\xi_1})\overline{U}_\varepsilon - g_u(u_{\xi_2},w_{\xi_2}) \overline{W}_\varepsilon&= - \Ue \\
\partial_t \overline{W}_\varepsilon - (1-\chi_\varepsilon)f_w(u_{\xi_1},w_{\xi_1})\overline{U}_\varepsilon - g_w(u_{\xi_2},w_{\xi_2})\overline{W}_\varepsilon
&= -\We,
\end{aligned}
\right.
\label{eq:mn3}
\end{equation}
with initial conditions $\overline{U}_\varepsilon(\cdot,T)=0$, $\overline{W}_\varepsilon(\cdot,T)=0$ and homogeneous Neumann boundary condition. Consider the change of variable: $z(\cdot,t) = \overline{U}_\varepsilon(\cdot,T-t)$, $y(\cdot,t) = \overline{W}_\varepsilon(\cdot,T-t)$ and
define $\widehat{U}_\varepsilon(\cdot,t) = \Ue(\cdot,T-t)$, $\widehat{W}_\varepsilon(\cdot,t) = \We(\cdot,T-t)$,
$\hat{f}_u(\cdot,t) = f_u(u_{\xi_1}(\cdot,T-t),w_{\xi_1}(\cdot,T-t))$, and analogously for $\hat{f}_w,\hat{g}_u,\hat{g}_w$. Hence $z$ and $y$ solve
\begin{equation}
\left\{
\begin{aligned}
\partial_t z - div(K_0 \nabla z) + (1-\chi_\varepsilon)\hat{f}_u z + \hat{g}_u y&= \widehat{U}_\varepsilon\quad \text{ in }Q_T, \\
\partial_t y + (1-\chi_\varepsilon)\hat{f}_w z + \hat{g}_w y &= \widehat{W}_\varepsilon\quad\text{ in }Q_T,\\
z(0)=0,\quad y(0)&=0\quad \text{ in }\Omega,\\
K_0\nabla z\cdot{\nu}&=0\quad\text{ on }\Gamma_T.
\end{aligned}
\right.
\label{eq:mn4}
\end{equation}
Since $\widehat{U}_\varepsilon,\widehat{W}_\varepsilon \in L^2(Q_T)$ and $\hat{f}_u, \hat{f}_w,\hat{g}_u,\hat{g}_w$ are bounded in $\overline{Q}_T$, by standard Faedo-Galerkin
technique we obtain that the solution $(z,y)$ of \eqref{eq:mn4} exists and is unique with the properties $z \in L^2(0,T;H^1(\Omega)) \cap L^\infty(0,T;L^2(\Omega))$,
$\partial_t z \in L^2(0,T;H^*)$, $y \in L^\infty(0,T;L^2(\Omega))$, $\partial_t y \in L^2(Q_T)$. Moreover
\[
\begin{aligned}
\norm{L^\infty(0,T;L^2(\Omega))}{z} &+ \norm{L^2(0,T;H^1(\Omega))}{z} + \norm{L^2(0,T;H^*)}{\partial_t z} + \norm{L^\infty(0,T;L^2(\Omega))}{y} + \norm{L^2(Q_T)}{\partial_t y}
\\ &\leq C \left(\norm{L^2(Q_T)}{\widehat{U}_\varepsilon} + \norm{L^2(Q_T)}{\widehat{W}_\varepsilon}\right) = C \left(\norm{L^2(Q_T)}{\Ue} + \norm{L^2(Q_T)}{\We}\right).
\end{aligned}
\]
Additional regularity of $z$ can be proved with an analogous argument as in \cite[Chapter 4, Theorem 9.1]{book:lady}, applied on the first equation in
\eqref{eq:mn4}. Indeed, by the regularity of $K_0$ and the square integrability of $\widehat{U}_\varepsilon - \hat{g}_u y$, we can conclude that
$z \in L^2(0,T;H^2(\Omega))$, $\partial_t z \in L^2(Q_T)$ and
\[
\norm{L^2(0,T;H^2(\Omega))}{z} + \norm{L^2(Q_T)}{\partial_t z} \leq C \norm{L^2(Q_T)}{\widehat{U}_\varepsilon - \hat{g}_u y}
\leq C\left( \norm{L^2(Q_T)}{\Ue} + \norm{L^2(Q_T)}{\We}\right).
\]
Moreover, multiplying the first two equations in \eqref{eq:mn4} respectively by $\partial_t z$ and $\partial_t y$ and integrating on $\Omega$, straightforward
computations allow to conclude that
$z \in L^\infty(0,T;H^1(\Omega))$ with
\[
\norm{L^\infty(0,T;H^1(\Omega))}{z} \leq C \left(\norm{L^2(Q_T)}{\Ue} + \norm{L^2(Q_T)}{\We}\right).
\]
By Sobolev inequality we obtain
\[
\begin{aligned}
\nabla z \in (L^2(0,T;L^6(\Omega)))^d,\,\, z \in L^6(Q_T),
\end{aligned}
\]
(all the norms being bounded by $\norm{L^2(Q_T)}{\Ue} + \norm{L^2(Q_T)}{\We}$; the same bounds hold on $\overline{U}_\varepsilon$). Thus, via
Nirenberg interpolation estimates (see \cite{nirenberg}), we get
\[
\begin{aligned}
 \norm{L^{10/3}(Q_T)}{\nabla \overline{U}_\varepsilon}  \leq C (\norm{L^2(Q_T)}{\Ue} + \norm{L^2(Q_T)}{\We}).
\end{aligned}
\]
Recalling also the previous results, this allows to conclude that, taking $p \in \left(2,\frac{10}{3}\right]$,
\begin{equation}
\norm{L^p(Q_T)}{\overline{U}_\varepsilon} +  \norm{L^p(Q_T)}{\nabla \overline{U}_\varepsilon} \leq C \left( \norm{L^2(Q_T)}{\Ue} + \norm{L^2(Q_T)}{\We}\right).
\label{eq:Lp}
\end{equation}
Let us multiply the equations of \eqref{eq:mn2} respectively by  $\overline{U}_\varepsilon,\overline{W}_\varepsilon$ and the first two equations of \eqref{eq:mn4}
respectively by $\Ue,\We$. Integrating on $(0,T)$ and summing the resulting identities it is easy to see that
\begin{equation}
\int_0^T \int_\Omega (\Ue^2+ \We^2) = \int_0^T \int_{\omegae} (K_0 - K_1) \nabla \ue \cdot \nabla \overline{U}_\varepsilon
+ \int_0^T \int_{\omegae} f(u,w)\overline{U}_\varepsilon.
\label{eq:aux1}
\end{equation}
Thanks to \eqref{eq:Lp} and  H\"older inequality the first term of the right-hand side can be bounded as follows
\[
\begin{aligned}
\int_0^T \int_{\omegae} (K_0 - K_1) \nabla \ue \cdot \nabla \overline{U}_\varepsilon
&\leq C \norm{L^q(\omegae \times (0,T))}{\nabla \ue} \norm{L^p(\omegae \times (0,T))}{\nabla \overline{U}_\varepsilon} \\
&\leq C \norm{L^q(\omegae \times (0,T))}{\nabla \ue} \left( \norm{L^2(Q_T)}{\Ue} + \norm{L^2(Q_T)}{\We} \right),
\end{aligned}
\]
where $p \in \left(2,\frac{10}{3}\right]$ and $q = \frac{p}{p-1} \in \left[\frac{10}{7}, 2\right)$. In addition, again by H\"older inequality,
\[
\begin{aligned}
\norm{L^q(\omegae \times (0,T))}{\nabla \ue} &\leq \norm{L^q(\omegae \times (0,T))}{\nabla u}
+ |\omegae|^{\frac{2-q}{2q}}\norm{L^q(\Omega \times (0,T))}{\nabla \Ue} \leq c |\omegae|^\frac{1}{q}.
\end{aligned}
\]
Furthermore, it is straightforward to see that the last term in \eqref{eq:aux1} can be bounded by
$c|\omegae|^\frac{1}{q} \norm{L^2(Q_T)}{\overline{U}_\varepsilon}$.
Finally, from \eqref{eq:aux1} we conclude that, for $q \in \left[\frac{10}{7},2\right)$,
\[
\norm{L^2(Q_T)}{\Ue}^2 + \norm{L^2(Q_T)}{\We}^2 \leq c |\omegae|^{\frac{1}{q}}\left( \norm{L^2(Q_T)}{\Ue} + \norm{L^2(Q_T)}{\We}\right).
\]
Hence
\[
\norm{L^2(Q_T)}{\Ue} + \norm{L^2(Q_T)}{\We} \leq c|\omegae|^\frac{1}{q} = c|\omegae|^{\frac{1}{2}+\beta},
\]
where $\beta \in \left(0,\frac{1}{5}\right]$.
\end{proof}

\subsection{Asymptotic expansion of boundary voltage}
In this section we prove the following result.
\begin{theorem}
For every sequence $\{\omega_{\varepsilon_n}\}$ with $|\omega_{\varepsilon_n}|>0$, $|\omega_{\varepsilon_n}| \rightarrow 0$, there exist a subsequence
(still denoted by $\omega_{\varepsilon_n}$), a Radon measure $\mu$ and a matrix-valued function $\mathcal{M} \in L^2(\Omega,d\mu;\R^{d\times d})$ such that,
for every pair $(\Phi,\Psi)\in (C^1(\overline{Q}_T),C(\overline{Q}_T))$ satisfying
\begin{equation}
		\left\{
		\begin{aligned}
			\partial_t \Phi + div(K_0 \nabla \Phi) - f_u(u,w) \Phi - g_u(u,w)\Psi &= 0 \\
			\partial_t \Psi - f_w(u,w)\Phi - g_w(u,w)\Psi &= 0 \\			
		\end{aligned}
		\right.
	\label{eq:PhiPsi}
	\end{equation}
	with the final conditions $\Phi(\cdot,T)=0$, $\Psi(\cdot,T)=0$, it holds
	\begin{equation}
	\int_0^T\!\!\!\!\int_{\partial \Omega} \!\!\!\!K_0 \nabla \Phi \cdot \normal (\uen - u)d\sigma
= |\omegaen| \int_0^T\!\!\!\!\int_\Omega \!\!\left[ \mathcal{M}(K_0-K_1)\nabla u \cdot \nabla \!\Phi + f(u,w)\Phi \right]d\mu + o(|\omegaen|).
	\label{eq:expansion}
	\end{equation}
\label{th:3}
\end{theorem}
In order to prove Theorem \ref{th:3}, we need to introduce the auxiliary functions $\vj$, $\venj$ solving
\begin{equation}
\left\{ \begin{aligned}
 div(K_0 \nabla \vj) &= F^{(j)} \quad &\textit{ in } \Omega ,\\
K_0\nabla \vj \cdot \normal &= f^{(j)}\quad &\textit{ on } \partial \Omega, \\
\int_{\partial \Omega} \vj &= 0, &
\end{aligned}\right.
\qquad \quad
\left\{ \begin{aligned}
  div(\Ken \nabla \venj) &= F^{(j)} \quad &\textit{ in } \Omega, \\
K_0\nabla \venj \cdot \normal &= f^{(j)} \quad &\textit{ on } \partial \Omega, \\
\int_{\partial \Omega} \venj &= 0, &
\end{aligned}\right.
\label{eq:vjvenj}
\end{equation}
with $F^{(j)} = \sum_{i} \frac{\partial[K_0]_{ij}}{\partial x_i}$ and $f^{(j)} = \sum_{i} [K_0]_{ij}\normal_i$. By the choice of $F^{(j)}$ and $f^{(j)}$ it holds $\vj = x_j - \int_{\partial \Omega} x_j$.
Moreover, $\vj$ and $\venj$ satisfy the following energy estimates (see \cite{art:capdebosq})
\begin{equation}
\norm{H^1(\Omega)}{\venj - \vj} \leq c |\omegaen|^\half, \qquad \norm{L^2(\Omega)}{\venj - \vj} \leq c |\omegaen|^{\half + \gamma}, \,\, \gamma >0.
\label{eq:energyv}
\end{equation}
We need also a preliminary lemma.
\begin{lemma}
For each $\phi \in C^1(\overline{Q}_T)$ such that $\phi(x,T)=0$, we have, as $|\omegaen| \rightarrow 0$,
\begin{equation}
\int_0^T \int_\Omega \frac{\chi_\omegaen}{|\omegaen|}(K_0-K_1)\nabla u \cdot \nabla \venj \phi
= \int_0^T \int_\Omega \frac{\chi_\omegaen}{|\omegaen|}(K_0 - K_1)\nabla \uen \cdot \nabla \vj \phi + o(1).
\label{eq:lemma4}
\end{equation}
\label{lem:asym}
\end{lemma}
\begin{proof}
Since $\vj$ and $\venj$ are the solutions of problems \eqref{eq:vjvenj}, we obtain
\[
\int_\Omega \Ken \nabla \venj \cdot \nabla \varphi = \int_\Omega K_0 \nabla \vj \cdot \nabla \varphi \qquad \forall \varphi \in H^1(\Omega).
\]
Take $\varphi = \Uen \phi$, being $\Uen = \uen - u$. By computation we get
\[
\int_\Omega \Ken \nabla \venj \cdot \nabla \Uen \phi + \int_\Omega \Ken \nabla \venj \cdot \nabla \phi \Uen =
\int_\Omega K_0 \nabla \vj \cdot \nabla \Uen\phi + \int_\Omega K_0 \nabla \vj \cdot \nabla \phi \Uen
\]
that we can rewrite as
\begin{equation}
\begin{aligned}
&\int_\Omega \Ken \nabla (\venj\phi) \cdot \nabla \Uen - \int_\Omega \Ken \venj \nabla \phi \cdot \nabla \Uen
+ \int_\Omega \Ken \nabla \venj \cdot \nabla \phi \Uen \\
&\quad = \int_\Omega K_0 \nabla (\vj\phi) \cdot \nabla \Uen - \int_\Omega K_0 \vj \nabla \phi \cdot \nabla \Uen
+ \int_\Omega K_0\nabla \vj \cdot \nabla \phi \Uen.
\end{aligned}
\label{eq:bob}
\end{equation}
Proposition \ref{prop:energy} and \eqref{eq:energyv} lead to
\[
\begin{aligned}
&\int_0^T\int_\Omega \Ken \venj \nabla \phi \cdot \nabla \Uen - \int_0^T\int_\Omega K_0 \vj \nabla \phi \cdot \nabla \Uen = \int_0^T\int_\Omega (\Ken - K_0) \nabla \uen \cdot \nabla \phi \vj\\
& \quad - \int_0^T\int_\Omega (\Ken - K_0) \nabla u \cdot \nabla \phi \venj + o(|\omegaen|),
\end{aligned}
\]
whereas, 
\[
\begin{aligned}
&\int_0^T\int_\Omega \Ken \nabla \venj \cdot \nabla \phi \Uen - \int_0^T\int_\Omega K_0\nabla \vj \cdot \nabla \phi \Uen \\
&\quad = \int_0^T\int_\Omega (\Ken - K_0) \nabla \venj \cdot \nabla \phi \Uen + \int_0^T\int_\Omega K_0 \nabla (\venj - \vj) \cdot \nabla \phi \Uen = o(|\omegaen|).
\end{aligned}
\]
Collecting the previous two relations in \eqref{eq:bob}, we conclude
\begin{equation}
\begin{aligned}
& \int_0^T\int_\Omega \Ken \nabla (\venj\phi) \cdot \nabla \Uen - \int_0^T\int_\Omega K_0 \nabla (\vj\phi) \cdot \nabla \Uen \\
& \quad = \int_0^T\int_\Omega (\Ken - K_0) \nabla \uen \cdot \nabla \phi \vj - \int_0^T\int_\Omega (\Ken - K_0) \nabla u \cdot \nabla \phi \venj + o(|\omegaen|).
\end{aligned}
\label{eq:vpart}
\end{equation}
It can be easily verified that, for every $\varphi \in H^1(\Omega)$, the following identities hold 
\[
\begin{aligned}
&\int_\Omega \partial_t \Uen \varphi + \int_\Omega \Ken \nabla \Uen \cdot \nabla \varphi + \int_\Omega (1-\chi_{\omegaen})(f(\uen,\wen)-f(u,w))\varphi \\
& \quad = \int_\Omega (K_0 - \Ken)\nabla u \cdot \nabla \varphi + \int_\Omega \chi_{\omegaen}f(u,w) \varphi ,\\
&\int_\Omega \partial_t \Uen \varphi + \int_\Omega K_0 \nabla \Uen \cdot \nabla \varphi + \int_\Omega (1-\chi_{\omegaen})(f(\uen,\wen)-f(u,w))\varphi \\
& \quad = \int_\Omega (K_0 - \Ken)\nabla \uen \cdot \nabla \varphi + \int_\Omega \chi_{\omegaen}f(u,w) \varphi. \\
\end{aligned}
\]
Taking $\varphi = \venj \phi$ in the first identity and $\varphi = \vj \phi$ in the second one, integrating in time and subtracting we obtain
\begin{equation}
\begin{aligned}
&\int_0^T \!\!\int_\Omega \partial_t \Uen(\venj-\vj)\phi + \int_0^T\!\!\int_\Omega \Ken\nabla\Uen\cdot\nabla(\venj\phi) - \int_0^T\!\!\int_\Omega K_0\nabla\Uen\cdot\nabla(\vj\phi)\\
& + \!\! \int_0^T\!\!\!\!\int_\Omega (1-\!\chi_{\omegaen})(f(\uen,\wen)-\!f(u,w))(\venj-\vj)\phi = \!\! \int_0^T \!\!\!\!\int_\Omega \chi_{\omegaen}f(u,w) (\venj-\!\vj)\phi \\
& +\int_0^T\!\!\int_\Omega (K_0 - \Ken)\nabla u \cdot \nabla (\venj\phi) - \int_0^T\!\!\int_\Omega (K_0 - \Ken)\nabla \uen \cdot \nabla (\vj\phi).
\end{aligned}
\label{eq:aux43}
\end{equation}
Via integration by parts, and taking advantage of the homogeneous initial and final conditions  satisfied respectively by $U_{\varepsilon_n}$ and $\phi$ it follows
\[
\begin{aligned}
&\left | \int_0^T\int_\Omega \partial_t \Uen(\venj-\vj)\phi \right | \leq \norm{C^1(Q_T)}{\phi}\norm{L^2(Q_T)}{\Uen}\norm{L^2(\Omega)}{\venj-\vj} = o(|\omegaen|).
\end{aligned}
\]
Analogous bounds can be proved for the last term in the left-hand side and the first term in the right-hand side of \eqref{eq:aux43},
thanks to the energy estimates of $\venj - \vj, \, \uen - u,\, \wen - w$ and the regularity of $f$. As a consequence, it holds that
\begin{equation}
\begin{aligned}
& \int_0^T\int_\Omega \Ken\nabla\Uen\cdot\nabla(\venj\phi) - \int_0^T\int_\Omega K_0\nabla\Uen\cdot\nabla(\vj\phi) \\
& = \quad \int_0^T\int_\Omega (K_0 - \Ken)\nabla u \cdot \nabla (\venj\phi) - \int_0^T\int_\Omega (K_0 - \Ken)\nabla \uen \cdot \nabla (\vj\phi) + o(|\omegaen|).
\end{aligned}
\label{eq:upart}
\end{equation}
In conclusion, by a combination of \eqref{eq:vpart} and \eqref{eq:upart},
\[
\begin{aligned}
&\int_0^T\int_\Omega (K_0 - \Ken)\nabla u \cdot \nabla (\venj\phi) - \int_0^T\int_\Omega (K_0 - \Ken)\nabla \uen \cdot \nabla (\vj\phi) \\
& \quad = \int_0^T\int_\Omega (\Ken - K_0) \nabla \uen \cdot \nabla \phi \vj - \int_0^T\int_\Omega (\Ken - K_0) \nabla u \cdot \nabla \phi \venj + o(|\omegaen|),
\end{aligned}
\]
which immediately entails the thesis.
\end{proof}

We are finally ready to prove the main result of this section.
\begin{proof}[Proof of Theorem \ref{th:3}]

Arguing as in \cite{art:capdebosq}, we deduce that there exist a Radon measure $\mu \in C(\overline{\Omega})^*$, a symmetric matrix $\mathcal{M} \in L^2(\Omega,d\mu;\R^{d\times d})$ and a sequence $\{\omegaen\}$ such that
\[
\frac{\chi_\omegaen}{|\omegaen|} \dx \rightarrow \dmu
\quad \text{and} \quad
\frac{\chi_\omegaen}{|\omegaen|}\frac{\partial \venj}{\partial x_i}\dx \rightarrow \mathcal{M}_{ij}\dmu
\]
in the weak* topology of $C(\overline{\Omega})$. Moreover, thanks to the regularity estimates for $u$ and the symmetries of the matrices $K_0$, $K_1$ and $\mathcal{M}$, we have, for every $j=1,\ldots,d$,
\begin{equation}
\frac{\chi_\omegaen}{|\omegaen|} (K_0-K_1)_{ik}\frac{\partial u}{\partial x_k}  \frac{\partial \venj}{\partial x_i}\dx \rightarrow  \mathcal{M}_{ji}(K_0-K_1)_{ik}\frac{\partial u}{\partial x_k}\dmu, \qquad \forall t \in (0,T),
\label{eq:Mu}
\end{equation}
which also implies  
\begin{equation}
\frac{\chi_\omegaen}{|\omegaen|} (K_0-K_1)_{ik}\frac{\partial u}{\partial x_k}  \frac{\partial \venj}{\partial x_i}\dx \dt \rightarrow  \mathcal{M}_{ji}(K_0-K_1)_{ik}\frac{\partial u}{\partial x_k}\dmu dt
\label{eq:Mut}
\end{equation}
in the weak$^*$ topology of $C(\overline{Q}_T)$.

On account of the energy estimates \eqref{eq:energyv}, straightforward computations show that for every $i,j=1,\ldots,d$ there exists a constant $C$, depending on the data and on $T$ but independent of $\varepsilon_n$, such that
\[ 
\left|\int_0^T \int_\Omega \frac{\chi_\omegaen}{|\omegaen|} (K_0 - K_1)_{ik} \frac{\partial \uen}{\partial x_k} \frac{\partial \vj}{\partial x_i} \dx \dt \right |\leq C.
\]
In particular, we define the limit measure $\nu_j$ in the weak$^*$ topology of $C(\overline{Q}_T)$) as follows
\begin{equation}
\frac{\chi_\omegaen}{|\omegaen|}\left[ (K_0 - K_1) \nabla \uen \right]_i \frac{\partial \vj}{\partial x_i} \dx \dt \rightarrow \dnuj 
\label{eq:nuj}
\end{equation}
Exploiting \eqref{eq:nuj} and \eqref{eq:Mut}, and recalling Lemma 3.3, we deduce
\begin{equation}
\int_0^T\int_{\Omega}\phi\dnuj = \int_0^T\int_{\Omega} \mathcal{M}_{ji}(K_0-K_1)_{ik}\frac{\partial u}{\partial x_k}\phi \dmu \qquad \forall \phi\in C^1(\overline{Q}_T) 
\label{eq:nuj2}
\end{equation}
and by the density of $C^1(\overline{Q}_T)$ in $C^0(\overline{Q}_T)$ we derive
\[
\dnuj=\mathcal{M}_{ji}(K_0-K_1)_{ik}\frac{\partial u}{\partial x_k} \dmu.
\]
Now, setting $\Uen = \uen-u$ and $\Wen = \wen-w$, and selecting $\Phi$ and $\Psi$ as test functions, we have
	\[
	\begin{aligned}
	&\int_\Omega \partial_t \Uen \Phi + \int_\Omega K_0 \nabla \Uen\cdot \nabla \Phi + \int_\Omega (1 - \chi_\omegaen)(f(\ue,\we)-f(u,w))\Phi
+ \int_\Omega \partial_t \Wen \Psi \\&+ \int_\Omega (g(\ue,\we)-g(u,w))\Psi = \int_\omegaen (K_0 - K_1)\nabla \uen \cdot \nabla \Phi
+ \int_\omegaen f(u,w)\Phi.
	\end{aligned}
	\]
Furthermore, since $(\Phi,\Psi)$ is a solution to \eqref{eq:PhiPsi}, it also holds:
\[
\begin{aligned}
 &\int_\Omega \partial_t \Phi \Uen - \int_\Omega K_0 \nabla \Phi \cdot \nabla \Uen - \int_\Omega (f_u(u,w)\Phi + g_u(u,w)\Psi)\Uen
 + \int_\Omega \partial_t \Psi\Wen \\&- \int_\Omega (f_w(u,w)\Phi + g_w(u,w)\Psi)\Wen = -\int_{\partial \Omega}K_0 \nabla \Phi \cdot \normal \Uen.
\end{aligned}
\]
Summing up the last two identities and integrating in time we observe that the terms involving the time derivatives vanish due to the initial and final conditions on $\Uen,\Wen,\Phi,\Psi$.
Consider now the non-linear terms containing $f$
\[
\begin{aligned}
& \int_0^T \int_\Omega (1-\chi_\omegaen)(f(\uen,\wen)-f(u,w))\Phi - \int_0^T\int_\Omega (f_u(u,w)\Uen + f_w(u,w)\Wen)\Phi \\
& = \int_0^T \int_\Omega (1-\chi_\omegaen)(f(\uen,\wen) - f(u,w) - f_u(u,w)\Uen - f_w(u,w)\Wen) \Phi \\
&+ \int_0^T \int_\omegaen (f_u(u,w)\Uen + f_w(u,w)\Wen)\Phi.
\end{aligned}
\]
Observe that by means of the Lagrange's mean value theorem and the Lipschitz-continuity of the functions $f_u,f_w$ with constants $L_{f_u}, L_{f_w}$ we have
\[
\begin{aligned}
&\int_0^T \int_\Omega (1-\chi_\omegaen)\left[ \left(f_u(u_{\xi_1}, w_{\xi_1})-f_u(u,w)\right) \Uen + \left(f_w(u_{\xi_1}, w_{\xi_1})-f_w(u,w)\right)\Wen \right] \Phi \\
& \quad \leq \int_0^T \int_\Omega (1-\chi_\omegaen)\left[L_{f_u} U_{\varepsilon_n}^2 + L_{f_u} \Uen\Wen + L_{f_w}\Uen\Wen + L_{f_w} W_{\varepsilon_n}^2 \right] \Phi \\
& \quad \leq c (\norm{L^2(Q_T)}{\Uen}^2 + \norm{L^2(Q_T)}{\Wen}^2) = o(|\omegaen|),
\end{aligned}
\]
whereas
\[
\begin{aligned}
& \int_0^T \int_\omegaen (f_u(u,w)\Uen + f_w(u,w)\Wen)\Phi \leq c \int_0^T \int_\Omega \chi_\omegaen (\Uen+\Wen) \\
& \quad \leq c |\omegaen|^\half (\norm{L^2(Q_T)}{\Uen} + \norm{L^2(Q_T)}{\Wen}) = o(|\omegaen|).
\end{aligned}
\]
Analogous estimates can be obtained for the terms containing $g$. Finally, 
\[
\begin{aligned}
\int_0^T \int_{\partial \Omega} K_0 \nabla \Phi \cdot \normal \Uen &= |\omegaen| \int_0^T \int_\Omega \frac{\chi_\omegaen}{|\omegaen|}\left[(K_0 - K_1)\nabla \uen \cdot \nabla \Phi + f(u,w)\Phi \right] + o(|\omegaen|) \\
&= |\omegaen| \sum_{i,j} \int_0^T \int_\Omega \frac{\chi_\omegaen}{|\omegaen|} [K_0 - K_1]_{ji} \frac{\partial \uen}{\partial x_i} \frac{\partial \Phi}{\partial x_j}  \\
& \quad + \int_0^T \int_\Omega \frac{\chi_\omegaen}{|\omegaen|}f(u,w)\Phi + o(|\omegaen|).
\end{aligned}
\]
Thanks to the regularity of $\Phi$, and employing the computed weak* limits, we derive
\[
\begin{aligned}
&\int_0^T \int_{\partial \Omega} K_0 \nabla \Phi \cdot \normal (\uen-u) \\
& \quad = |\omegaen| \left( \sum_{i,j,k} \int_0^T\int_\Omega \mathcal{M}_{kj} [K_0 - K_1]_{ki}\frac{\partial u}{\partial x_i} \frac{\partial \Phi}{\partial x_j} d\mu\ dt + \int_0^T\int_\Omega f(u,w)\Phi d\mu \ dt \right)+ o(|\omegaen|) \\
& \quad = |\omegaen| \int_0^T\int_\Omega \left[ \mathcal{M}(K_0-K_1) \nabla u \cdot \nabla \Phi + f(u,w)\Phi\right] d\mu\ dt + o(|\omegaen|).
\end{aligned}
\]
which concludes the proof.
\end{proof}

\section{Reconstruction algorithm: a topology optimization approach}
\label{algorithm}
The asymptotic expansion provided in Theorem \ref{th:3} allows to describe the perturbation of the electrical potential on the boundary of the domain due to the presence of a small conductivity inhomogeneity $\omega_{\varepsilon}$. 
In order to derive a reconstruction algorithm for problem \eqref{eq:inv}, we introduce the mismatch functional $J$
\begin{equation}
J(\omega) = \half \int_0^T \int_{\partial \Omega} (u_\varepsilon - u_{meas})^2,
\label{eq:J}
\end{equation}
where $u_\varepsilon=u_{\omega_\varepsilon}$ solves the perturbed problem \eqref{eq:perturbed} in the presence of an ischemic region $\omega_\varepsilon$. We now prove that the functional $J$ restricted to the class of inclusions satisfying \eqref{eq:separ}, \eqref{eq:small} admits an asymptotic expansion with respect to the size of the inclusion. Moreover, as it is shown in Theorem \ref{th:4}, the first-order term of the expansion (which will be denoted as $G$, the topological gradient of $J$) can be computed by solving the unperturbed problem and a suitable adjoint problem. 
In particular, we restrict ourselves to inclusions $\omegae$ satisfying \eqref{eq:separ} and of the form
\begin{equation}
\omegae = \{z+ \varepsilon B\},
\label{eq:omegaez}
\end{equation}
where $B$ is a bounded, smooth set containing the origin. 

We have the following
\begin{theorem}
Consider a family $\{\omegae\}$ satisfying \eqref{eq:separ} and \eqref{eq:omegaez}. Then, there exists a matrix $\mathcal{M}$ (which may depend on $z$, $B$, $K_0$ and $K_1$) such that, as $\varepsilon\to 0$,
\[
\begin{aligned}
J(\omegae) = &J(0) +  |\omegae| \int_0^T \left[ \mathcal{M}(K_0(z)-K_1(z))\nabla u(z,t) \nabla \Phi(z,t) + f(u(z,t),w(z,t))\Phi(z,t) \right]\dt \\
&+ o(|\omegae|),
\end{aligned}
\]
where $(u,w)$ solves \eqref{eq:formaforte} and $(\Phi,\Psi)$ is the solution of the \textit{adjoint problem}:
\begin{equation}
\left\{
\begin{aligned}
\partial_t \Phi + div(K_0 \nabla \Phi) - f_u(u,w)\Phi - g_u(u,w)\Psi &= 0 \quad &\textit{ in } Q_T,\\
K_0 \nabla \Phi \cdot \normal &= u - u_{meas} \quad &\textit{ on } \Gamma_T,\\
\partial_t \Psi - f_w(u,w)\Phi - g_w(u,w)\Psi &= 0 \quad &\textit{ in } Q_T,\\
\Phi(\cdot,T) = 0 \qquad \Psi(\cdot,T) &= 0 \quad &\textit{ in } \Omega.
\end{aligned}
\right.
\label{eq:adjoint}
\end{equation}
\label{th:4}
\end{theorem}

\begin{proof}
First of all, we note that, since Lemma 3.3 holds under the weaker assumption  on the test functions, $\phi\in W^{1,\infty}({Q}_T)$, then the validity of Theorem \ref{th:3} can be easily extended when $(\Phi,\Psi)\in (W^{1,\infty}({Q}_T), C(\overline{Q}_T))$. We now show that the solution to (\ref{eq:adjoint}) enjoys this regularity. 
In fact this can be ensured by a lifting argument, since the boundary datum $u-u_{meas} \in C^{2+\alpha,1+\alpha/2}(\Gamma_T)$. Estimates of the $W^{1,\infty}(Q_T)$ norm of $\Phi$ can be obtained by flattening the boundary and by reflection arguments, see for example \cite[Chapter 7, Theorem 2.2]{book:lady} for a two-dimensional case.

It is straightforward to verify that
\begin{equation}
\begin{aligned}
J(\omegae) - J(0) &= \int_0^T\int_{\partial \Omega} (\ue - u)(u - u_{meas}) + \half \int_0^T\int_{\partial \Omega} (\ue - u)^2.
\end{aligned}
\label{eq:solito}
\end{equation}
The first term in the right-hand side of \eqref{eq:solito} involves the boundary condition of the adjoint problem, and in particular it can be written as $\int_0^T\int_{\partial \Omega} K_0 \nabla \Phi \cdot \normal (\ue - u)$.
We now apply Theorem \ref{th:3} in order to conclude that
\[
	\int_0^T\!\!\int_{\partial \Omega} (\ue - u)(u - u_{meas}) = |\omegae| \int_0^T \!\!\int_\Omega \left[ \mathcal{M}(K_0-K_1)\nabla u \cdot \nabla \Phi + f(u,w)\Phi \right]d\mu + o(|\omegae|).
\]
According to assumption \eqref{eq:omegaez}, the limit measure $\mu$ is independent of the choice of the subsequence and is equal to $\delta_z$, the Dirac measure centered in $z$. Finally, the second term in the right-hand side of \eqref{eq:solito} is $o(|\omegae|)$ by means of Lemma \ref{lem:boundary} below.
\end{proof}

\begin{remark}
	By Assumption 2.1  and by the homogeneous boundary conditions on $\ue$ we can conclude, using standard local regularity results on parabolic equations applied in a neighbourhood of $\partial\Omega$, that $\ue \in C^{2+\alpha,1+\alpha/2}(\Gamma_T)$
\end{remark}

\begin{lemma}
Let $\{\omegae\}$ satisfy \eqref{eq:separ} and \eqref{eq:omegaez}. Then,
\[
  \int_0^T\int_{\partial \Omega} (\ue - u)^2 = o(|\omegae|).
\]
\label{lem:boundary}
\end{lemma}
\begin{proof}
	By the same argument as in the proof of Theorem \ref{th:4} and by Remark 4.2,
	\[
  \int_0^T\int_{\partial \Omega} (\ue - u)^2 = |\omegae|\int_0^T\int_\Omega \left[\mathcal{M}(K_0 - K_1)\nabla u\cdot \nabla \Theta + f(u,w)\Theta \right]d\mu + o(|\omegae|),
\]
where $(\Theta,\Xi)$ satisfies
\begin{equation}
\left\{
\begin{aligned}
\partial_t \Theta + div(K_0 \nabla \Theta) - f_u(u,w)\Theta - g_u(u,w)\Xi &= 0 \quad &\textit{ in } Q_T,\\
K_0 \nabla \Theta \cdot \normal &= \ue - u \quad &\textit{ on } \Gamma_T,\\
\partial_t \Xi - f_w(u,w)\Theta - g_w(u,w)\Xi &= 0 \quad &\textit{ in } Q_T,\\
\Theta(\cdot,T) = 0 \qquad \Xi(\cdot,T) &= 0 \quad &\textit{ in } \Omega.
\end{aligned}
\right.
\label{eq:test}
\end{equation}
and $(\Theta,\Xi)\in(W^{1,\infty}({Q}_T), C(\overline{Q}_T)) $.

We now focus on proving that for a compact set $K$ such that $\omegae\subset K$ and satisfying $dist(K,\partial \Omega) \geq d_0$ one has
\begin{equation}
\norm{L^2(0,T;H^3(K))}{\Theta} \leq \norm{L^2(0,T;H^1(\Omega))}{\ue-u},
\label{eq:H3}
\end{equation}
and, as a consequence of a Sobolev immersion, (\ref{eq:H3})  implies (here $d=2,3$)
\[
\norm{L^2(0,T;W^{1,\infty}(K))}{\Theta} \leq \norm{L^2(0,T;H^1(\Omega))}{\ue-u} \leq |\omegae|^\half.
\]
By a combination of \eqref{eq:separ} and \eqref{eq:omegaez} we can ensure that the support of the measure $\mu = \delta_z$ is contained in $K$. This would entail that
\[
\int_0^T\int_\Omega \left[\mathcal{M}(K_0 - K_1)\nabla u\cdot \nabla \Theta + f(u,w)\Theta \right]d\mu \leq c |\omegae|^\half,
\]
hence $\int_0^T\int_{\partial \Omega} (\ue - u)^2 = O(|\omegae|^{3/2}) = o(|\omegae|)$.
\par
In order to prove \eqref{eq:H3}, by standard Faedo-Galerkin technique applied to the linear system \eqref{eq:test}, one has
\begin{equation}
\norm{L^\infty(0,T;L^2(\Omega))}{\Theta} + \norm{L^\infty(0,T;L^2(\Omega))}{\Xi} + \norm{L^2(0,T;H^1(\Omega))}{\Theta} \leq c \norm{L^2(0,T;L^2(\partial \Omega))}{\ue - u}.
\label{eq:L2bound}
\end{equation}
Moreover, according to \eqref{eq:test}, $\Theta$ is the solution of the following problem
\[
\left\{
\begin{aligned}
\partial_t \Theta + div(K_0 \nabla \Theta) - f_u(u,w)\Theta &= g_u(u,w)\Xi \quad &\textit{ in } Q_T,\\
K_0 \nabla \Theta \cdot \normal &= \ue - u \quad &\textit{ on } \Gamma_T,\\
\Theta(\cdot,T) &= 0 \quad &\textit{ in } \Omega.
\end{aligned}
\right.
\]
Analogously to what stated in \cite[Chapter 4, Theorem 9.1]{book:lady}, it holds that
\[
\norm{L^2(0,T;H^2(\Omega))}{\Theta} \leq C \left(\norm{L^2(Q_T)}{g_u(u,w)\Xi} + \norm{L^2(0,T;H^{1/2}(\partial \Omega))}{\ue - u}\right)
\]
and in conclusion, by the regularity of $g, u, w$ and by \eqref{eq:L2bound}
\begin{equation}
\norm{L^2(0,T;H^2(\Omega))}{\Theta} \leq C \norm{L^2(0,T;H^{1/2}(\partial \Omega))}{\ue - u}.
\label{eq:H2bound}
\end{equation}
Similar estimates can be derived also for $\Xi$ since $f,g \in C^3(\R^2)$ and $u,w \in C^{2,1}(\overline{Q_T})$. This last property is guaranteed on $u$ by Theorem \ref{th:1}, whereas it can be extended to $w$ since we consider $w_0 \in C^2(\overline\Omega)$. Now, by computing the solution of the third equation in \eqref{eq:test} in closed form, we can easily verify that
\begin{equation}
\norm{L^2(0,T;H^2(\Omega))}{\Xi} \leq C \norm{L^2(0,T;H^2(\Omega))}{\Theta} \leq C \norm{L^2(0,T;H^{1/2}(\partial \Omega))}{\ue - u}.
\label{eq:Xi}
\end{equation}
Consider the change of variable $s = T-t$ and denote by $\widehat{\Theta}(\cdot,s) = \Theta(\cdot,T-t), \widehat{\Xi}(\cdot,s) = \Xi(\cdot,T-t)$. We now focus on the first equation in \eqref{eq:test} and compute the second derivatives of each term. Then, $V \defeq \frac{\partial^2 \widehat{\Theta}}{\partial x_i \partial x_j}$ satisfies the following equation in a weak sense:
\begin{equation}
\partial_{s} V -div(K_0 \nabla V) + f_u V = R,
\label{eq:VR}
\end{equation}
being
\[
\begin{aligned}
&R =  div\left(\frac{\partial K_0}{\partial x_i} \nabla \frac{\partial \widehat{\Theta}}{\partial x_j} \right) + div\left(\frac{\partial K_0}{\partial x_j} \nabla \frac{\partial \widehat{\Theta}}{\partial x_i} \right) + div\left(\frac{\partial^2 K_0}{\partial x_i \partial x_j} \nabla \widehat{\Theta} \right)\\
& -\left(f_{uuu}\frac{\partial u}{\partial x_i}\frac{\partial u}{\partial x_j} + f_{uuw}\frac{\partial u}{\partial x_i} \frac{\partial w}{\partial x_j} + f_{uu} \frac{\partial^2 u}{\partial x_i \partial x_j} + f_{uwu}\frac{\partial w}{\partial x_i}\frac{\partial u}{\partial x_j} + f_{uww}\frac{\partial w}{\partial x_i} \frac{\partial w}{\partial x_j} + f_{uw} \frac{\partial^2 w}{\partial x_i \partial x_j} \right)\widehat{\Theta} \\
&-\left( f_{uu}\frac{\partial u}{\partial x_i} + f_{uw}\frac{\partial w}{\partial x_i}\right)\frac{\partial \widehat{\Theta}}{\partial x_j} - \left( f_{uu}\frac{\partial u}{\partial x_j} + f_{uw}\frac{\partial w}{\partial x_j}\right)\frac{\partial \widehat{\Theta}}{\partial x_i} \\
& - \left(g_{uuu}\frac{\partial u}{\partial x_i}\frac{\partial u}{\partial x_j} + g_{uuw}\frac{\partial u}{\partial x_i} \frac{\partial w}{\partial x_j} + g_{uu} \frac{\partial^2 u}{\partial x_i \partial x_j} + g_{uwu}\frac{\partial w}{\partial x_i}\frac{\partial u}{\partial x_j} + g_{uww}\frac{\partial w}{\partial x_i} \frac{\partial w}{\partial x_j} + g_{uw} \frac{\partial^2 w}{\partial x_i \partial x_j} \right)\widehat{\Xi} \\
& - \left( g_{uu}\frac{\partial u}{\partial x_i} + g_{uw}\frac{\partial w}{\partial x_i}\right)\frac{\partial \widehat{\Xi}}{\partial x_j} - \left( g_{uu}\frac{\partial u}{\partial x_j} + g_{uw}\frac{\partial w}{\partial x_j}\right)\frac{\partial \widehat{\Xi}}{\partial x_i} - g_u \frac{\partial^2 \widehat{\Xi}}{\partial x_i \partial x_j}.
\end{aligned}
\]
Let $\Omega_1$ be an open subset of $\Omega$, such that $\omegae\subset K\subset \Omega_1 \subset \subset \Omega$. By interior regularity results it holds that $(\widehat{\Theta},\widehat{\Xi})$ is smooth on $\Omega_1\times[0,T]$,  which entails by the initial conditions that
\begin{equation}
\frac{\partial^2\widehat{\Theta}}{\partial x_i \partial x_j}(x,0)=0,\quad\forall x\in \Omega_1.
\label{eq:Theta0}
\end{equation}
According to the smoothness of $f,g,u,w$ and to the $H^2$ bounds \eqref{eq:H2bound} and \eqref{eq:Xi}, we get
\begin{equation}
\norm{L^2(0,T;H^{-1}(\Omega_1))}{R} \leq C \norm{L^2(0,T;H^{1/2}(\partial \Omega))}{\ue - u}.
\label{eq:Rbound}
\end{equation}
In fact, all the terms in $R$ except the first three belong to $L^2(Q_T)$, with norm bounded by $\norm{L^2(0,T;H^2(\Omega))}{\Theta}$ or by $\norm{L^2(0,T;H^2(\Omega))}{\Xi}$. Set $H:=div\left(\frac{\partial K_0}{\partial x_i} \nabla \frac{\partial \widehat{\Theta}}{\partial x_j} \right)$, then, for any $v \in H^1_0(\Omega_1)$,
\[
 |\langle H,v\rangle_{\star}|= \left|\int_{\Omega_1} \frac{\partial K_0}{\partial x_i} \nabla \frac{\partial \widehat{\Theta}}{\partial x_j} \cdot \nabla v \right|\leq C\norm{H^2(\Omega)}{\widehat{\Theta}}\norm{H^1(\Omega_1)}{v}.
\]
where $\langle \cdot,\cdot \rangle_{\star}$ indicates the duality pairing between the involved spaces. Integrating in time last relation we finally get
\[
\norm{L^2(0,T;H^{-1}(\Omega_1))}{H} \leq C\norm{L^2(0,T;H^{1/2}(\partial \Omega))}{\ue - u}.
\]
With similar arguments we can estimate also the second and third term in the definition of the function $R$ leading finally to (\ref{eq:Rbound}).

Consider now a test function $\zeta \in C^\infty_C(\Omega_1)$, $0 \leq \zeta \leq 1$, such that $\zeta = 1$ in $K$.
According to \eqref{eq:VR}, by simple computations we verify that $\widetilde{V} = V\zeta$ satisfies
\begin{equation}
    \partial_s \widetilde{V} - div(K_0 \nabla \widetilde{V}) + f_u \widetilde{V} = \widetilde{R},
    \label{eq:VRt}
\end{equation}
being $\widetilde{R} = \zeta R - 2 div(V K_0 \nabla \zeta) + V div(K_0 \nabla \zeta)$. It holds that $div(V K_0 \nabla \zeta) \in L^2(0,T;H^{-1}(\Omega_1))$, and
\[
\norm{L^2(0,T,H^{-1})}{div(V K_0 \nabla \zeta)} \leq C \norm{L^2(Q_T)}{V},
\]
hence we observe that
\begin{equation}
\norm{L^2(0,T,H^{-1}(\Omega_1))}{\widetilde{R}} \leq C ( \norm{L^2(Q_T)}{V} + \norm{L^2(0,T,H^{-1}(\Omega_1))}{R}) \leq C \norm{L^2(0,T;H^{1/2}(\partial \Omega))}{\ue - u}
 \label{eq:Rtilde}
\end{equation}
By standard Faedo-Galerkin argument (see, e.g., \cite[Chap. 18, Par. 3, Theorem 3]{book:dautraylions}, we verify that
\begin{equation}
\norm{L^2(0,T;H^1_0(\Omega_1))}{\widetilde{V}}\leq C \norm{L^2(0,T,H^{-1}(\Omega_1))}{\widetilde{R}}.
    \label{eq:FG}
\end{equation}
Finally, by the definition of $\zeta$, combining \eqref{eq:FG} and \eqref{eq:Rtilde} we conclude that
\[
\norm{L^2(0,T;H^1(K))}{V} \leq \norm{L^2(0,T;H^1_0(\Omega_1))}{\widetilde{V}}\leq C \norm{L^2(0,T;H^{1/2}(\partial \Omega))}{\ue - u}.
\]

The thesis is now proved thanks to the trace inequality and the energy estimates.
\end{proof}
Theorem \ref{th:4} gives a representation formula for the topological gradient $G$, the first order term in the expansion of the mismatch functional $J$. In analogy to \cite{art:BMR,art:bccmr}, we propose a \textit{one-step} reconstruction Algorithm \ref{al:top} for the identification of small inclusions satisfying \eqref{eq:omegaez}.
\begin{algorithm}[H]
\begin{algorithmic}
	\REQUIRE $u_{meas}(x,t) \ \forall \, x \in \partial \Omega$, $t \in (0,T)$.
	\ENSURE approximated centre of the inclusion, $\bar{z}$
	\begin{itemize}
		\item compute $(u,w)$ by solving \eqref{eq:formaforte};
		\item compute $(\Phi,\Psi)$ by solving \eqref{eq:adjoint};
		\item determine the topological gradient $G$ of $J$ according to Theorem \ref{th:4};
		\item find $\bar{z}$ s.t. $G(\bar{z}) \leq G(z) \quad \forall \, z \in \Omega$.
	\end{itemize}
 \end{algorithmic}
\caption{Reconstruction of a small-size inclusion}
\label{al:top}
\end{algorithm}
\begin{remark}
	The polarization tensor $\mathcal{M}$ can be computed in an explicit form, e.g., when the shape $B$ of the inclusion is a disk, even though it is not straightforward in the current anisotropic case. Since $\mathcal{M}$ depends on $K_0$ and $K_1$ and on the shape of the set $B$, which we assume to be a disk, it varies according to the space variable $z$. In order to compute $\mathcal{M}(z) = \mathcal{M}(K_0(z),K_1(z),B)$, we first use \cite[Lemma 4.30]{book:ammari-kang} to transform $K_0$ and $K_1$ in diagonal matrices. By a rescaling of the spatial variables, we can reduce to the case in which one of the two diagonalized tensors is the identity, which allows to apply \cite[Proposition 4.31]{book:ammari-kang}. We remark that the anisotropic rescaling entails that the original circle $B$ is transformed in an ellipse of known semiaxes.
\end{remark}

\section{Numerical results}
\label{results}
In this section we describe the implementation of Algorithm \ref{al:top} and report some numerical results in order to show its effectiveness. In particular, we set our experiments in a two-dimensional idealized geometry, representing a horizontal section of the ventricles. The application of the model on a three-dimensional geometry would be equivalently possible, and will be object of future studies, employing advanced numerical analysis techniques in order to tackle the increased computational effort. We rely on synthetic data, i.e., we solve the monodomain system in presence of a prescribed ischemic region and then use the value of the solution on the boundary (or a portion of it) as an input for the reconstruction algorithm. In order to prevent \textit{inverse crimes}, we employ different numerical settings for the synthetic data generation and for the solution of the unperturbed and adjoint problems, required for the reconstruction algorithm. In particular, we adopt a much more refined discretization, both in space and in time, for the simulation of the synthetic data.
\par
In the following experiments we assume to measure the voltage only on a portion of the heart. As outlined in Section \ref{intro}, measurements of the voltage can be acquired on the inner surface of a ventricle by intracavitary measurements, or, alternatively, we might be able to compute a map of the electrical potential on the epicardium starting from ECG data. This does not affect the reconstruction procedure described in Algorithm \ref{al:top}, apart from the definition of the adjoint problem \eqref{eq:adjoint}, which now prescribes oblique boundary conditions involving $u_{meas}-u$ on the portion of the boundary where the measurements are acquired, and homogeneous Neumann conditions elsewhere.

\subsection{Finite Element approximation}
In order to numerically approximate the solution of the monodomain model, we rely on a Galerkin Finite Element scheme, introducing a tessellation of the domain consisting of triangular elements. In particular, we adopt two different meshes for the solution of the unperturbed (and adjoint) problem and for the generation of the synthetic measurements, the latter being more refined especially close to the boundary of the prescribed ischemic region.
\par
Moreover, in order to reproduce the anisotropic behavior of the conductivity coefficients $K_0$ and $K_1$, we consider the presence of fibers within the domain. We adopt a procedure analogous to the one reported in \cite{Quarteroni2016} for the generation of the fiber directions, resorting on the solution of a Laplace problem with suitable boundary conditions.
Once the direction of the fibers $\vec{e_f}(x) = \vec{e}_1(x)$ is defined within $\Omega$, as well as the transmural vector field $\vec{e_n}(x) = \vec{e}_2(x)$ (obtained by a clockwise rotation of $90^\circ$ of $\vec{e_f}(x)$), the definition of the conductivity tensors is

\[
K_0(x) = k_{0,1} \vec{e_f}(x) \otimes \vec{e_f}(x) + k_{0,2} \vec{e_n}(x) \otimes \vec{e_n}(x), \quad
K_1(x) = k_{1,1} \vec{e_f}(x) \otimes \vec{e_f}(x) + k_{1,2} \vec{e_n}(x) \otimes \vec{e_n}(x).
\]
The numerical mesh for the solution of the background problem and the directions of the fibers are reported in Figure \ref{fig:setup}.
\begin{figure}[h!]
			\centering
			\subfloat[Mesh]{
		    	\includegraphics[width=0.4\textwidth]{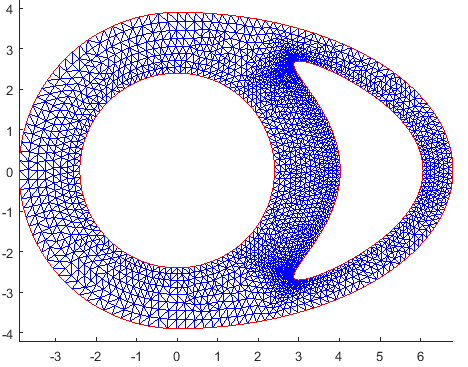}
			}
			\hspace{0.1\textwidth}
			\subfloat[Fibers orientation]{
		    	\includegraphics[width=0.4\textwidth]{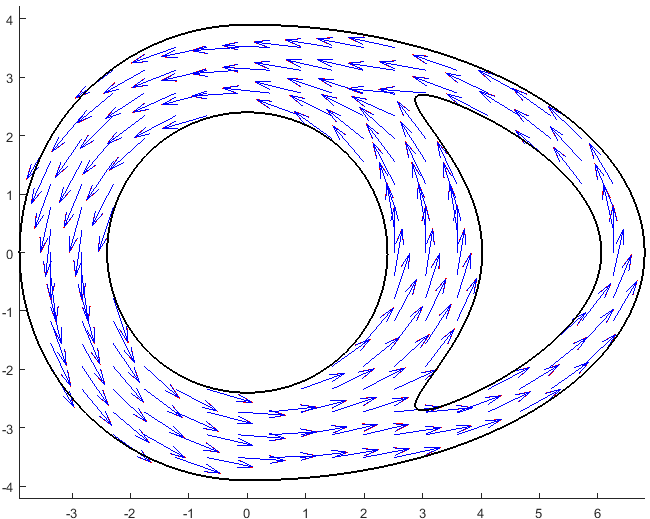}
			}
	\caption{Setup of the numerical test cases}
	\label{fig:setup}
	\end{figure}
	As already mentioned in the introduction, the best choice of the conductivity tensor $K_0$ to be employed in the numerical simulations of the monodomain model is given by the harmonic mean of the intracellular and extracellular conductivity tensors $D_i$ and $D_e$ appearing in the bidomain model that can be expressed by
\begin{equation}
	K_0 = D_e (D_e + D_i)^{-1} D_i.
\label{eq:}
\end{equation}
Exploiting the fact that $D_e$ and $D_i$ have the same eigenvectors (again defined by the fiber and transmural directions), we can therefore compute the eigenvalues of $K_0$ (and analogously of $K_1$) as the harmonic mean of the ones of $D_i$ and $D_e$ as reported, e.g., in \cite[Table 8.1]{book:pavarino}.
The values of the main parameters involved in the numerical simulations are reported in Table \ref{tab:param}.
\begin{table}%
\centering
\begin{tabular}{c|c|c|c|c|c|c}
$k_{0,1}$ & $k_{0,2}$ & $k_{1,1}$ & $k_{1,2}$ & $A$ & $a$ & $\epsilon$ \\ \hline
$1.200$ & $0.2538$ & $0.2308$ & $0.0062$ & $8$ & $0.15$ & $0.05$
\end{tabular}
\caption{Values of the main parameters}
\label{tab:param}
\end{table}
\par
The numerical solution of the background problem \eqref{eq:formaforte} relies on a Newton-Galerkin scheme. The spatial discretization is performed thanks to the P1-finite element space, i.e., the space of the continuous functions over $\Omega$ which are linear polynomials when restricted on each element of the mesh. The temporal discretization is done via an implicit Euler scheme. This leads to solve a nonlinear problem at each timestep, which is performed via a Newton iterative algorithm. More details on the solver can be found in \cite{art:RV}, where a thorough convergence analysis is performed. 

\subsection{Reconstruction of small inclusions}	
We study the effectiveness of the reconstruction algorithm in identifying the position of small ischemic regions within the cardiac tissue. For, we employ synthetic measurements of the perturbed boundary voltage in presence of ischemias located in different sections of the cardiac tissue: in the left ventricle, in the septum, or in the right ventricle. In each simulation, we consider circular ischemic regions of radius ranging from $1mm$ to $2mm$. In Figure \ref{fig:recon} we report the contour plot of the topological gradient $G$ in four different cases, and superimpose a black line representing the boundary of the ischemia that we aim to reconstruct.
\begin{figure}[h!]
			\centering \vspace{-0.5cm}
			\subfloat[Small inclusion in the septum \newline from epicardiac measurements]{
		    	\includegraphics[width=0.5\textwidth]{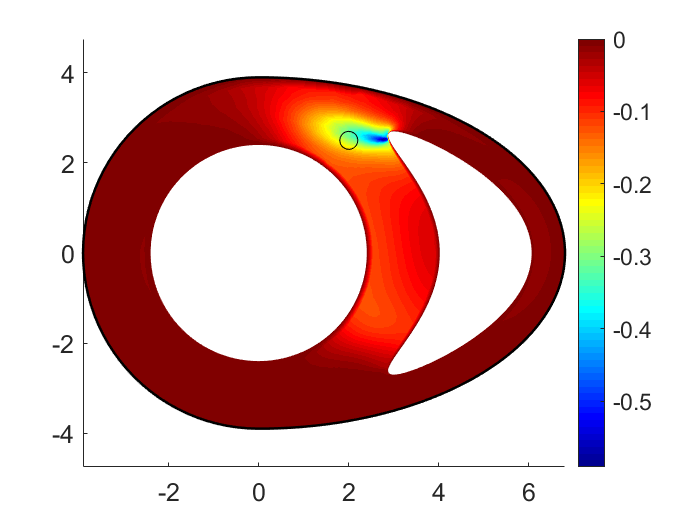}
			}
			\subfloat[Small inclusion in the Left Ventricle \newline from left-endocariac measurements]{
		    	\includegraphics[width=0.5\textwidth]{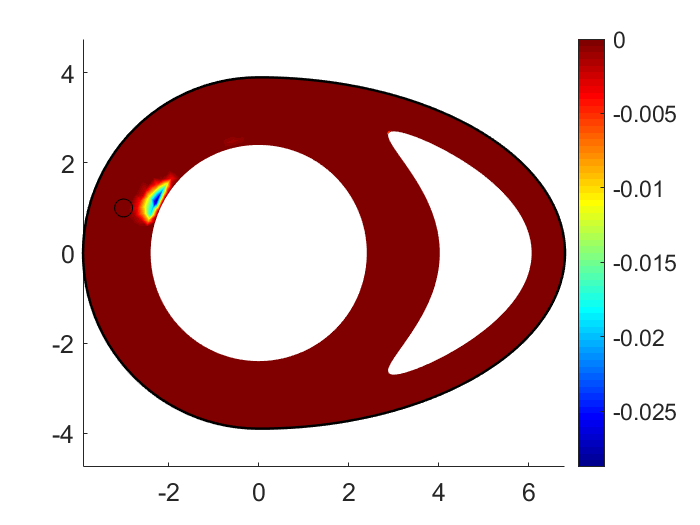}
			}\\
			\subfloat[Small inclusion in the Right Ventricle \newline from right-endocariac measurements]{
		    	\includegraphics[width=0.5\textwidth]{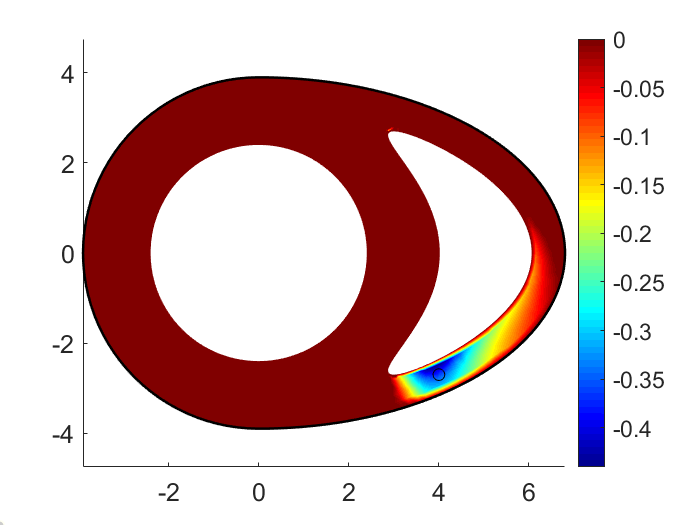}
			}
			\subfloat[Two small inclusions \newline from epicardiac measurements]{
		    	\includegraphics[width=0.5\textwidth]{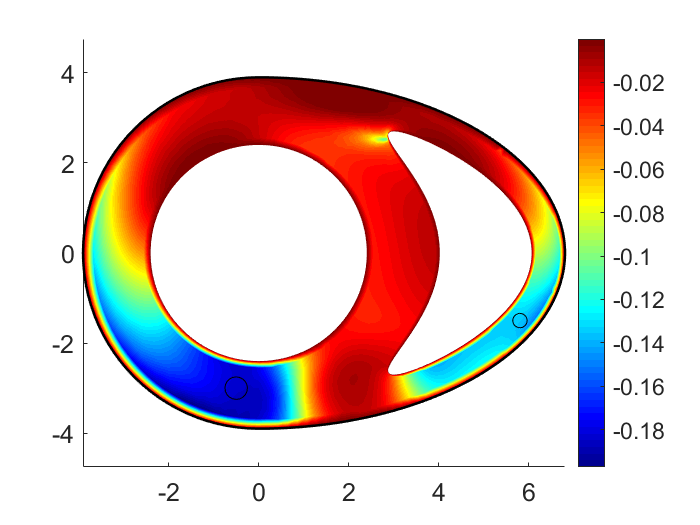}
			}
	\caption{Reconstruction of small inclusions}
	\label{fig:recon}
	\end{figure}
From an analysis of Figure \ref{fig:recon} we can deduce that the topological gradient always attains its negative minimum value close to the real position of the ischemia, and the accuracy of the localization may vary according to the position of the ischemia. We achieve significant results also in the case of multiple inclusions. When more than one ischemic region is present, the topological gradient shows a local minimum close to each region.
As depicted in the captions, some experiments employ the knowledge of the voltage on the epicardium and some on one endocardiac surface. The reconstruction appears fairly accurate in both cases. We remark that the algorithm is always effective when using epicardiac measurements, but fails in detecting ischemic regions located in a ventricle whenever the measurements are acquired on the inner surface of the other one. Nevertheless, even in those cases the technique avoids false positive detection, i.e., the topological gradient does not show significant local minima in wrong locations.

\subsection{Reconstruction from noisy measurements}
We now focus on the performance of the algorithm in presence of noisy measurements, in order to assess the stability of the algorithm with respect to small perturbations of the boundary data of the form
\begin{equation}
\widetilde{u}_{meas}(x,t) = u_{meas}(x,t) + \rho \eta(x,t),
\label{eq:noise}
\end{equation}
where $\eta(x,t)$ is a standard Gaussian random variable for each point $x$ and instant $t$, and $\rho \in [0,1]$ is the noise level.
\begin{figure}[h!]
			\centering \vspace{-0.5cm}
			\subfloat[Measurements without noise]{
		    	\includegraphics[width=0.5\textwidth]{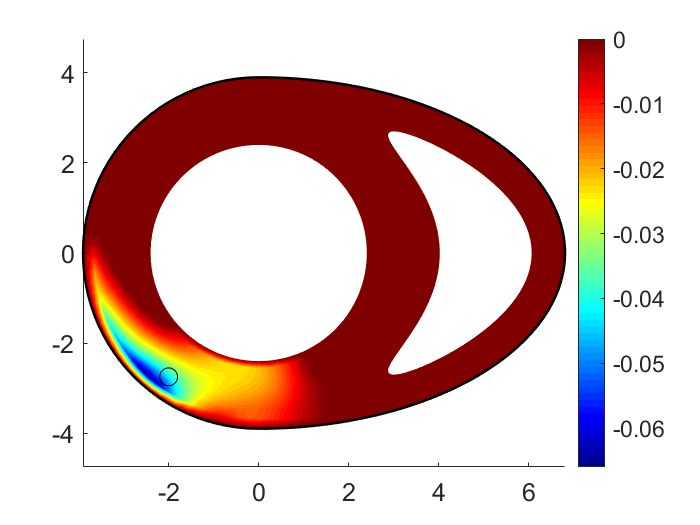}
			}
			\subfloat[Noise level: $5\%$]{
		    	\includegraphics[width=0.5\textwidth]{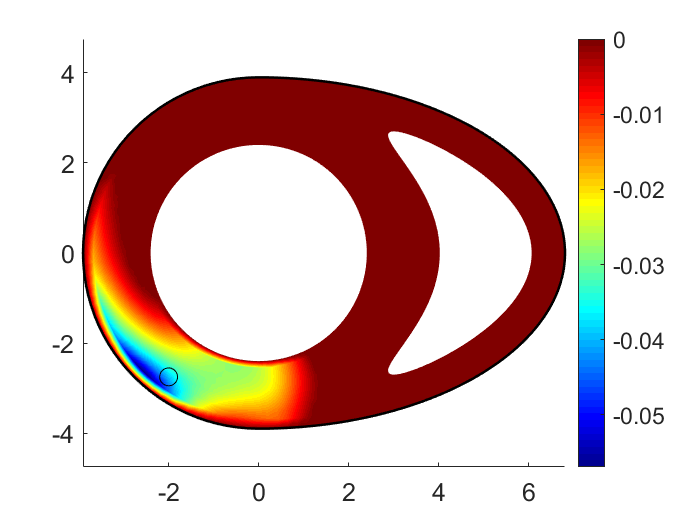}
			}\\
			\subfloat[Noise level: $10\%$]{
		    	\includegraphics[width=0.5\textwidth]{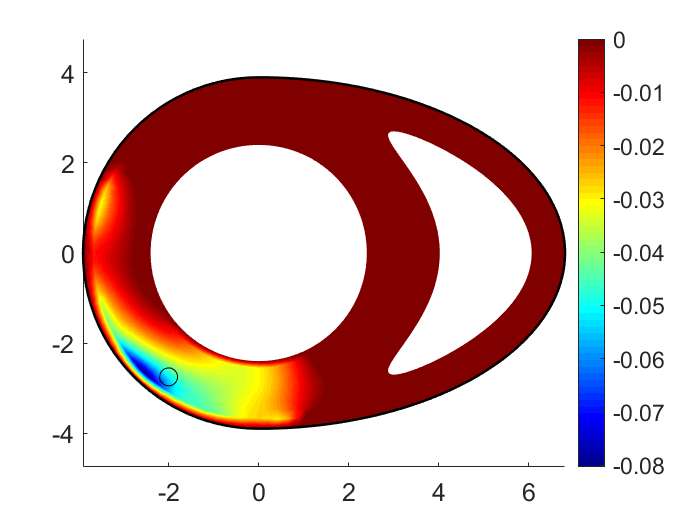}
			}
			\subfloat[Noise level: $15\%$]{
		    	\includegraphics[width=0.5\textwidth]{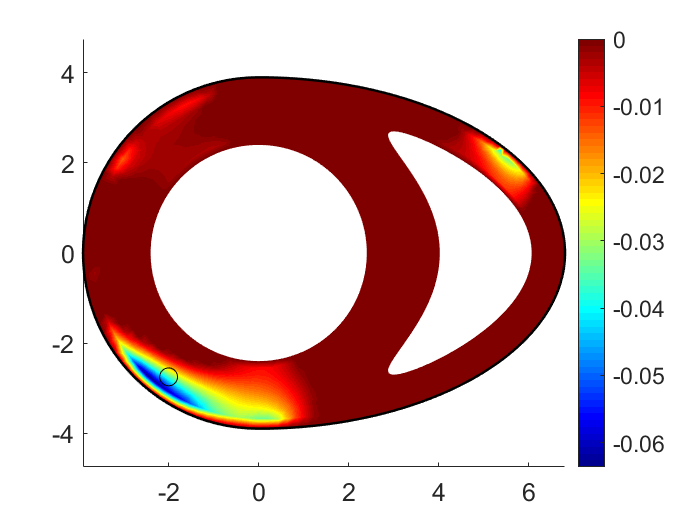}
			}
	\caption{Reconstruction in presence of noisy epicardiac measurements}
	\label{fig:noise}
	\end{figure}
	In Figure \ref{fig:noise} we report the contour plot of the topological gradient computed with noiseless measurements and compare it with the ones obtained with growing levels of noise. The algorithm shows to be stable even under significantly corrupted measurements: up to the level $\rho = 0.15$, the negative region in correspondence to the exact position is clearly identifiable.
	
	\subsection{Reconstruction of large inclusions}
	We eventually remark that the proposed algorithm produces significant results also when applied to measurements associated to large ischemic regions. The identification of larger regions could be performed by means of iterative reconstruction algorithms, as the one proposed in \cite{art:BRV} for a semilinear elliptic problem. Nevertheless, as it is shown in Figure \ref{fig:big}, the information coming from the topological gradient could be a suitable initial guess for such iterative algorithms.
\begin{figure}[h!]
			\centering \vspace{-0.5cm}
			\subfloat[Reconstruction of a large ischemia \newline from epicardiac measurements]{
		    	\includegraphics[width=0.5\textwidth]{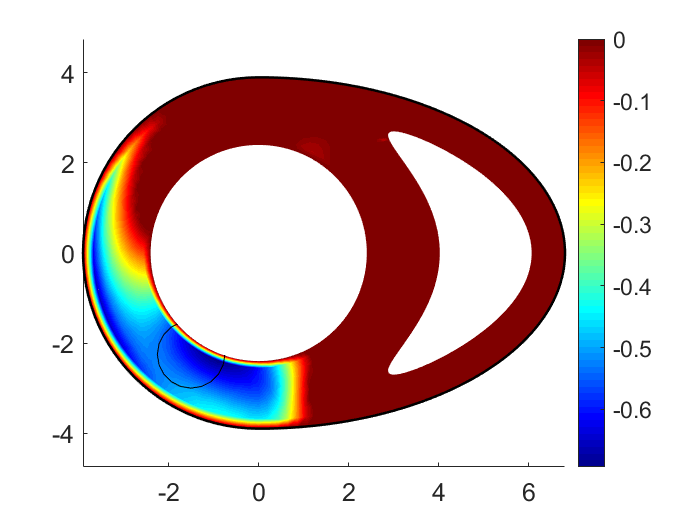}
			}
			\subfloat[Reconstruction of a large ischemia \newline from right-endocardiac measurements]{
		    	\includegraphics[width=0.5\textwidth]{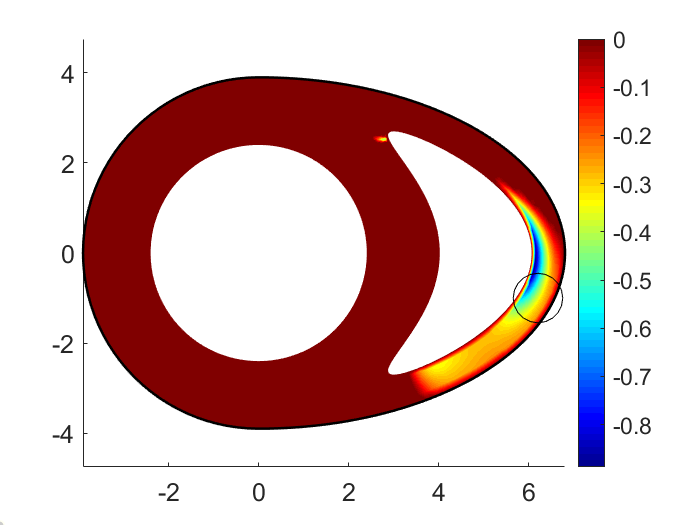}
			}
	\caption{Reconstruction of large ischemic regions}
	\label{fig:big}
	\end{figure}

\section{Conclusions}
In this paper, we have studied the identification of small ischemic regions in the cardiac tissue, represented by discontinuous alterations in the conductivity and in the nonlinear reaction term of the monodomain model, taking advantage of the measurement of the voltage on the boundary of the heart. We have extended the existing results regarding the well-posedness of the direct problem and derived a rigorous asymptotic expansion of the perturbed boundary voltage, which allows to formulate an effective reconstruction algorithm based on topological optimization.
\par
We foresee several significant extensions of the presented analysis, both from an analytical and a numerical viewpoint. The coupling of the monodomain model of the heart with a passive conductor model for the surrounding torso (see \cite{art:boulakia}) would enable us to make use of ECG data for reconstruction purposes, possibly comparing the results with the ones obtained in \cite{art:lyka},\cite{wang2013inverse} with a stationary version of the bidomain model. 
\par
The effectiveness of the reconstruction algorithm should also be tested in a three-dimensional setting; in such a context, due to the high computational cost of the numerical simulation, a Reduced Order Modeling approach could be considered, as recently analysed in \cite{art:pagani}. 
\par
Finally, the detection of arbitrarily large ischemic regions can be tackled, exploiting an analogous strategy as the one adopted in \cite{art:BMR} (possibly reducing the computational effort by employing an adaptive solver of the monodomain system, as proposed in \cite{art:RV}).

\section*{Acknowledgments}
The authors are grateful to Roberto Gianni for his useful suggestions. The authors thank the New York University of Abu Dhabi for its kind hospitality. The work of the authors has been supported by GNAMPA (Gruppo Nazionale per l'Analisi Matematica, la Probabilità e le loro Applicazioni). The numerical simulations presented in this work have been performed thanks to the MATLAB library redbKIT, \cite{redbKIT}.

\printbibliography

\end{document}

%% file: setup_extended.tex
\usepackage[margin=1.25in]{geometry}
\usepackage{algorithm}
\usepackage{amsthm}
\usepackage{caption}                         
\usepackage[utf8]{inputenc}
\usepackage{amssymb}
\usepackage{verbatim}
\usepackage{graphicx}
\usepackage{epstopdf}
\usepackage{indentfirst}                     
\usepackage[english]{babel}
\usepackage{subfig}                          
\usepackage{chngpage, calc}                  
\usepackage[font=small]{quoting}             
\usepackage[autostyle]{csquotes}             
\usepackage{amsmath}
\usepackage{mathtools}
\usepackage{bookmark}
\usepackage{mdframed}
\usepackage{xcolor}
\usepackage{enumerate}
\usepackage{hyperref}                        
\usepackage{algorithmic}
\usepackage[style=numeric,backend=bibtex]{biblatex}
\usepackage{empheq}
\usepackage{tikz}

\graphicspath{{Figures/}}
\bibliography{bibliography}
\numberwithin{equation}{section}

\newcommand{\omegae}{\omega_{\varepsilon}}

\newcommand{\ue}{u_{\varepsilon}}

\newcommand{\we}{w_{\varepsilon}}

\newcommand{\chie}{\chi_{\varepsilon}}

\newcommand {\half}{ \frac{1}{2} }

\newcommand{\N}{\mathbb{N}} 
\newcommand{\R}{\mathbb{R}}

\newcommand{\norm}[2]{{\left\| #2 \right\|}_{#1}}

\newcommand{\normal}{\nu}
\newcommand{\defeq}{\mathrel{\mathop:}=}
\newcommand{\eqdef}{=\mathrel{\mathop:}}

\newtheorem{theorem}{Theorem}
\newtheorem{proposition}{Proposition}
\newtheorem{lemma}{Lemma}

\newtheorem{assumption}{Assumption}
\newtheorem{remark}{Remark}

\graphicspath{{Figures/}} 

\newcommand{\omegaen}{{\omega_{\varepsilon_n}}}
\newcommand{\uen}{{u_{\varepsilon_n}}}
\newcommand{\wen}{{w_{\varepsilon_n}}}
\newcommand{\uo}{u_\omega}
\newcommand{\wo}{w_\omega}
\newcommand{\vj}{{v^{(j)}}}
\newcommand{\venj}{{v^{(j)}_{\varepsilon_n}}}
\newcommand{\Ue}{{U_{\varepsilon}}}
\newcommand{\We}{{W_{\varepsilon}}}
\newcommand{\Uen}{{U_{\varepsilon_n}}}
\newcommand{\Wen}{{W_{\varepsilon_n}}}
\newcommand{\Ken}{{K_{\varepsilon_n}}}
\newcommand{\dx}{\text{d}x}
\newcommand{\dt}{\text{d}t}
\newcommand{\dmu}{\text{d}\mu}
\newcommand{\dnuj}{\text{d}\nu_j}

